\documentclass[oneside,final, letterpaper]{ucr} 

\usepackage{amsmath,verbatim,amssymb,amsfonts,amscd, graphicx}
\usepackage{graphics}
\usepackage{tikz-cd}
\usepackage{bbm}
\usepackage{changepage}
\usepackage{scrextend}
\usepackage{paralist}
\usepackage{combelow}
\usepackage{setspace}
\usepackage{etoolbox}

\newcommand*{\RR}{\mathbb{R}}
\newcommand*{\ZZ}{\mathbb{Z}}
\newcommand*{\QQ}{\mathbb{Q}}
\newcommand*{\CC}{\mathbb{C}}
\newcommand*{\FF}{\mathbb{F}}
\newcommand*{\DD}{\mathbb{D}}
\newcommand*{\NNN}{\mathbb{N}}
\newcommand*{\pp}{\mathfrak{p}}

\newcommand*{\OO}{\mathcal{O}}
\newcommand*{\HH}{\mathcal{H}}

\newcommand*{\HHH}{\mathbb{H}}
\newcommand*{\ZZZ}{\mathfrak{Z}}
\newcommand*{\lag}{\mathfrak{g}}

\newcommand*{\CCC}{\mathcal{C}}
\newcommand*{\KK}{\mathcal{K}}
\newcommand*{\MM}{\mathcal{M}}
\newcommand*{\NN}{\mathcal{N}}

\newcommand*{\BK}{\mathcal{B}(\mathcal{K})}
\newcommand*{\BH}{\mathcal{B}(\mathcal{H})}

\newcommand*{\Lc}{^\circ \! L}
\newcommand{\at}[2][]{#1|_{#2}}

\newenvironment{psmallmatrix}
  {\left(\begin{smallmatrix}}
  {\end{smallmatrix}\right)}

\usepackage[backend=biber, 
style=alphabetic, 
]{biblatex}
\bibliography{thesis.bib}    
\AtBeginBibliography{\small}

\begin{document}


\title{Two New Settings for Examples of von Neumann Dimension}
\author{Lauren Chase Ruth}
\degreemonth{June}
\degreeyear{2018}
\degree{Doctor of Philosophy}
\chair{Professor Feng Xu}
\chairtwo{Professor Alain Valette}
\othermembers{Professor Vyjayanthi Chari}
\numberofmembers{3}
\field{Mathematics}
\campus{Riverside}

\maketitle
\copyrightpage{}
\approvalpage{}

\degreesemester{Spring}

\begin{frontmatter}

\begin{acknowledgements}
I am lucky to have two advisors, Feng Xu and Alain Valette, and I thank them both for their continuous support.  Feng Xu suggested that I read \cite{toa}, which contains the theorem that is the starting point for this project; and through him, I have been incredibly fortunate to become connected to the world of operator algebras and to attend major workshops and conferences, including the Von Neumann Algebras Trimester Program during Summer 2016 at the Hausdorff Research Institute for Mathematics, which turned out to be crucial for my development as a mathematician.  It was there that I met Alain Valette, who, via email, has watched this project grow from 2 pages to its current length, caught numerous serious errors, offered multiple helpful suggestions, and answered countless questions, always responding with not just ``yes'' or ``no,'' but with an argument or counterexample. This project could not have been completed without his guidance.  In addition, I feel supported by a network of mathematicians who have advanced my career, whether by inviting me to give a talk, securing me a workshop spot, teaching me new mathematics, or simply making me feel welcome in the mathematical community:  Nina Yu, Mike Hartglass, Dietmar Bisch, Brent Nelson, Ian Charlesworth, Rolando de Santiago, Ben Hayes, Henry Tucker, Julia Plavnik, Hans Wenzl, Dave Penneys, Ionu\cb{t} Chifan, Thomas Sinclair, Paul Garrett, Steven Spallone, Manish Mishra, A.\ Raghuram, Claus S\o rensen, Chenxu Wen, and Jesse Peterson.  I would like to thank the following institutions for their hospitality:  HIM, Bonn, Germany; MSRI, California, USA; TSIMF, Sanya, China; and IISER, Pune, India.  Finally, thanks to the math grad students at UCR who covered my TA duties while I attended mathematical events, and to the Math Dept.\ and GSA for providing me with additional travel funding.    
\end{acknowledgements}

\begin{dedication}
\null\vfil
{\large
\begin{center}
For Mom, Dad, Ali, and Julia.
\end{center}}
\vfil\null
\end{dedication}

\begin{abstract}
Let $G=PSL(2,\mathbb{R})$, let $\Gamma$ be a lattice in $G$, and let $\mathcal{H}$ be an irreducible unitary representation of $G$ with square-integrable matrix coefficients. A theorem in \cite{toa} states that the von Neumann dimension of $\mathcal{H}$ as a $R\Gamma$-module is equal to the formal dimension of the discrete series representation $\mathcal{H}$ times the covolume of $\Gamma$, calculated with respect to the same Haar measure. We prove two results inspired by this theorem.  First, we show there is a representation of $R\Gamma_2$ on a subspace of cuspidal automorphic functions in $L^2(\Gamma_1 \backslash G)$, where $\Gamma_1$ and $\Gamma_2$ are lattices in $G$; and this representation is unitarily equivalent to one of the representations in \cite{toa}. Next, we calculate von Neumann dimensions when $G$ is $PGL(2,F)$, for $F$ a local non-archimedean field of characteristic $0$ with residue field of order not divisible by 2; $\Gamma$ is a torsion-free lattice in $PGL(2,F)$, which, by a theorem of Ihara, is a free group; and $\mathcal{H}$ is the Steinberg representation, or a depth-zero supercuspidal representation, each yielding a different dimension. 
\end{abstract}

\tableofcontents
\end{frontmatter}


\chapter{Introduction}

This dissertation is about representing free group factors (and other II$_1$ factors arising from lattices in $PSL(2,\RR)$ and $PGL(2,F)$) on certain Hilbert spaces that are important from the standpoint of the representation theory of $SL(2,\RR)$ and $GL(2,F)$.  The original research content is in Chapters 9 and 14, where we prove the main results of the dissertation.  

In the present chapter, we outline the content of the other chapters.  

In Chapter 2, we review important subgroups of and the Lie algebra of $SL(2,\RR)$.

In Chapter 3, we introduce discrete series representations of $SL(2,\RR)$, working in the setting of $SU(1,1)$.

In Chapter 4, we define automorphic forms, focusing on cusp forms.  

In Chapter 5, we relate cusp forms to discrete series representations.

In Chapter 6, we record a conjecture which would imply the non-vanishing of a certain kind of Poincar\'e series.  (This chapter is not essential to proving the main results in Chapters 9 and 14.)

In Chapter 7, we recall basic facts about von Neumann algebras (specifically, finite factors).

In Chapter 8, we review the proof of the theorem in \cite{toa} that served as the starting point for this project, and we give examples of von Neumann dimension.

In Chapter 9, we prove a new version of the theorem in \cite{toa} involving two lattices instead of one, using a lemma on surjective intertwiners between representations of factors, a well-known theorem relating the multiplicity of discrete series representations to dimensions of spaces of cusp forms, and formulas for dimensions of spaces of cusp forms.  

In Chapter 10, we introduce non-archimedean local fields (structure, extensions, and characters), denoted by $F$, focusing on the case when the field characteristic is $0$ and residue field is of order not divisible by $2$.  

In Chapter 11, we review important subgroups of $GL(2,\FF_q)$ and $GL(2,F)$, as well as different normalizations of Haar measure.

In Chapter 12, we list some propositions necessary for understanding induced representations and the smooth dual.  

In Chapter 13, we calculate formal dimensions of the Steinberg representation and a depth-zero supercuspidal representation with respect to different Haar measures using standard facts from the representation theory of $\pp$-adic groups.

In Chapter 14, we compute von Neumann dimension (from the formula in \cite{toa}) in the setting of $PGL(2,F)$ by dealing carefully with Haar measure.  

In Chapter 15, we summarize part of the local Jacquet-Langlands correspondence (about formal dimensions of discrete series representations).  

In Chapter 16, we explain our plan to expand this dissertation and break it up into two papers.   
\chapter{Structure of $SL(2,\RR)$} \label{setup}

Most of the material in this chapter can be found in Chapter 1 of \cite{vog}, Sections 2 and 4 of \cite{baf}, Chapter 1 of \cite{miy}, parts of \cite{bum}, and \cite{tou}.  

Let 
$$G=SL(2,\RR)= \left\lbrace g = \begin{pmatrix}
a & b \\
c & d
\end{pmatrix} \, : \, a,b,c,d \in \RR, \, ad-bc=1 \right\rbrace,
$$
and let
$$K=SO(2)=
\left\lbrace k_\theta = \begin{pmatrix}
  \cos \theta     &  \sin \theta\\ 
  - \sin \theta & \cos \theta 
\end{pmatrix} \, :\,  \theta \in [0, 2\pi) \right\rbrace.
$$
$G$ acts on the upper-half plane $\HHH$ by linear fractional transformations, and this action extends to an action on $\RR \cup \left\lbrace \infty \right\rbrace$, the boundary of $\HHH$ in the Riemann sphere.  We identify $G/K$ with $\HHH$ by sending $g \in G$ to $g(i)$.   

Let $z \in \HHH$.  For $g \in G$, we have the \textit{automorphy factor} 
\begin{equation} \label{autfac}
j(g, z) = cz+d, 
\end{equation}
which satisfies the cocycle identity 
 \begin{align} \label{coc}
 j(g g', z) = j(g,g'z)j(g',z) \qquad (g, g' \in G).
 \end{align} 
For any function $f(x)$ on $G$, and any $y \in G$, let the operators $\rho(y)$ and $\lambda(y)$ be defined by
\begin{align*}
\rho(y) f(x) &=  f(xy) \\ 
\lambda(y) f(x) &= f(y^{-1}x).
\end{align*}
We may identify $K$ with the group ${ e^{i \theta} }$ of modulus-1 complex numbers. Its characters are its 1-dimensional representations, 
$$\chi_m : k_\theta \mapsto e^{i m \theta}.$$ 
Because $K$ fixes $i$, the cocycle identity (\ref{coc}) gives
 \begin{align} \label{Kcoc}
 j(k k', i) = j(k,i)j(k',i) \qquad (k, k' \in K)
 \end{align} 
so $j( \cdot , i)$ is a character of $K$.  In fact,
\begin{align} \label{Kchar}
j(k,i)=\chi_{-1}(k).
\end{align}
We say a function $f$ on $G$ is of \textit{right K-type m} if 
$$\rho(k) f(x) = \chi_m(k) f(x) \qquad (x \in G, \, k \in K)$$
We say a function $f$ on $G$ is \textit{K-finite on the right} if the right  translations $f(xk)$ span a finite-dimensional vector space. \textit{Left K-type m} and \textit{K-finite on the left} are defined in the same way, for $\lambda(k)$.  Any $K$-finite function $f$ is a finite sum $f=\sum f_i$, where $f_i$ is of some type $m$.  

The Lie algebra of $G$ is
$$\lag = \left\lbrace  \begin{pmatrix}
a & b \\
c & d
\end{pmatrix} \, : \, a,b,c,d \in \RR, \, a+d=0 \right\rbrace.
$$
A basis for $\lag$ is given by 
$$ 
H=\begin{pmatrix} 
  1     &  0\\ 
  0 & -1 
\end{pmatrix}, \qquad
X=\begin{pmatrix} 
  0     &  1\\ 
  0 & 0
\end{pmatrix}, \qquad
Y=\begin{pmatrix} 
  0    &  0\\ 
  1 & 0 
\end{pmatrix},
$$
satisfying the commutation relations
$$ [H,X]=2X, \qquad [H,Y]=-2Y, \qquad [X,Y] = H. $$
Every $W \in \lag$ is associated with the 1-parameter subgroup of $G$ 
$$   t \mapsto e^{tW} = \sum_{n=0} ^\infty \frac{t^n W^n}{n !} \qquad (t \in \RR).$$
We may identify every $W \in \lag$ with the differential operator on $C^\infty(G)$ defined by
$$ Wf(x) =  \frac{d}{dt} f \left( x e ^{tW} \right) \!\at[\Big]{t=0} 
\qquad ( f \in C^ \infty (G), \, x \in G). $$
Note that these operators are invariant under the left action of $G$.  These operators generate the algebra of \textit{all} left-invariant differential operators over $\CC$.  The resulting algebra is isomorphic to the universal enveloping algebra of $\lag$, denoted by $\mathcal{U}(\lag)$, and we identify the two algebras.  (If we had defined the differential operator identified with $W$ to act by ``$e ^{tW} x$'' instead of ``$x e ^{tW}$'' --- right-invariant, instead of left-invariant --- the $W$ would generate an algebra \textit{anti}-isomorphic to $\mathcal{U}(\lag)$.)

Let $\ZZZ$ denote the center of $\mathcal{U}(\lag)$.  $\ZZZ$ consists of left-invariant differential operators that are also right-invariant, and it is generated over $\CC$ by the Casimir element 
$$\CCC=\frac{1}{2}H^2+XY+YX.$$  
We say a function $f$ on $G$ is \textit{$\ZZZ$-finite} if it is annihilated by an ideal $J$ of finite codimension in $\ZZZ$.  This is equivalent to the existence of a non-constant polynomial in the Casimir element which annihilates $f$.

Let $\DD = \left\lbrace w \in \CC : \left| w \right| < 1 \right\rbrace $, the unit disc.  Let 
$$T= \begin{pmatrix} 
  1     &  -i\\ 
  1 & i 
\end{pmatrix}.$$
As a linear fractional transformation on $\HHH$, $T$ maps $i$ to $0$.  Under the automorphism of $SL(2,\CC)$ given by $g \mapsto TgT^{-1}$, $SL(2,\RR)$ transforms onto the group of conformal mappings of $\DD$,
$$G_\DD=SU(1,1)= \left\lbrace g = \begin{pmatrix}
a & b \\
\bar{b} & \bar{a}
\end{pmatrix} : \vert a \vert ^2 - \vert b \vert ^2 = 1 \right\rbrace ,
$$
$K$ transforms onto
$$K_\DD= \left\lbrace \begin{pmatrix}
a & 0 \\
0 & \bar{a}
\end{pmatrix} : \vert a \vert = 1 \right\rbrace
= \left\lbrace k_\theta = \begin{pmatrix}
e^{i\theta} & 0 \\
0 & e^{-i \theta}
\end{pmatrix} : \theta \in [0,2\pi) \right\rbrace ,
$$
and we may identify $G_\DD / K_\DD$ with $\DD$.  We have
\begin{align} \label{HtoD}
T (gz) = (T g T^{-1}) (Tz) \qquad (g \in G, \, z \in \HHH).
\end{align}
Let $\chi_m$ be the character of $K_\DD$ defined by $\chi_m(k_\theta)=e^{im\theta}$.  If $f$ is of right $K$-type $m$ on $G$, then the function 
$$ r(g)=f( T^{-1} g T ) \qquad (g \in G_\DD) $$
is of right $K_\DD$-type $m$ on $G_\DD$.

Let $w \in \DD$.  As in (\ref{autfac}), for $g \in G_\DD$, we have the automorphy factor
\begin{align*}
j_\DD(g, w) = \bar{b} w + \bar{a},
\end{align*}
which, as in (\ref{coc}), satisfies the cocycle identity 
 \begin{align} \label{cocD}
 j_\DD(g g', w) = j_\DD(g,g'w)j_\DD(g',w) \qquad (g, g' \in G).
 \end{align} 
Because $K_\DD$ fixes $0$, the cocycle identity (\ref{cocD}) gives
 \begin{align} \label{KDcoc}
 j_\DD(k k', 0) = j_\DD(k,0)j_\DD(k',0) \qquad (k, k' \in K_\DD),
 \end{align} 
so $j_\DD( \cdot , 0)$ is a character of $K_\DD$.  In fact, we have something better than (\ref{Kchar}):
\begin{align} \label{KDchar}
j_\DD(k,w)=\chi_{-1}(k) \qquad (k \in K_\DD, \, w \in \DD).
\end{align}
A basis for the Lie algebra $\lag_\DD$ of $G_\DD$ is given by 
$$ 
iH=\begin{pmatrix} 
  i     &  0\\ 
  0 & -i 
\end{pmatrix}, \qquad
X+Y=\begin{pmatrix} 
  0     &  1\\ 
  1 & 0
\end{pmatrix}, \qquad
i(X-Y)=\begin{pmatrix} 
  0    &  i\\ 
  -i & 0 
\end{pmatrix}
$$
satisfying the commutation relations
\begin{align} \label{comm}
[iH,X+Y]=2i(X-Y), \qquad [iH,i(X-Y)]=-2(X+Y), \qquad [X+Y,i(X-Y)] = -2iH.
\end{align}
Note $K_\DD=e^{\RR i H}$, so $iH$ spans the Lie algebra of $K_\DD$. 
Let $\lag_\DD^\CC$ denote the complexification of $\lag_\DD$.  It has a basis given by
$$ 
H=\begin{pmatrix} 
  1     &  0\\ 
  0 & -1 
\end{pmatrix}, \qquad
X=\begin{pmatrix} 
  0     &  1\\ 
  0 & 0
\end{pmatrix}, \qquad
Y=\begin{pmatrix} 
  0    &  0\\ 
  1 & 0 
\end{pmatrix},
$$
satisfying the commutation relations
$$ [H,X]=2X, \qquad [H,Y]=-2Y, \qquad [X,Y] = H. $$
\begin{lemma} \label{hfmf}
$f$ on $G_\DD$ is of right $K_\DD$-type $m$ $\iff$ $iHf=imf$ $\iff$ $Hf=mf$.
\end{lemma}
\begin{proof}
Suppose $f$ is of right $K$-type $m$. Then
\begin{equation} 
iHf(x) = \frac{d}{dt} f \left( x e ^{tiH} \right) \!\at[\Big]{t=0} = 
\frac{d}{dt} f \left( x k \right) \!\at[\Big]{t=0} = 
\frac{d}{dt} \left( f \left( x  \right) \chi _m(k) \right) \!\at[\Big]{t=0} = 
\frac{d}{dt} \left( f \left( x  \right) e^{imt} \right) \!\at[\Big]{t=0} = 
imf(x).
\end{equation}
And conversely, 
$$ k_\theta = e^{\theta iH} = \sum_{n=0} ^\infty \frac{\theta^n (iH)^n}{n !},$$
so if $iHf=imf$,
$$iHf(x) = \frac{d}{dt} f \left( x e ^{tiH} \right) \!\at[\Big]{t=0} = 
im f(x) $$
which implies
$$\frac{d}{dt} f \left( x e ^{tiH} \right) \!\at[\Big]{t=0} = \frac{d}{dt} f \left( x e ^{tim} \right) \!\at[\Big]{t=0} $$
and 
$$\rho(k_\theta) f(x) = e^{im\theta} f(x) = \chi_m(k_\theta)f(x)$$ 
whence $f$ is of right $K$-type $m$.  
\end{proof}

\begin{lemma} \label{ef}
$f$ on $G_\DD$ is of right $K_\DD$-type $m$ $\Rightarrow$ $Xf$ is of right $K_\DD$-type $m+2$, and $Yf$ is of right $K_\DD$-type $m-2$.
\end{lemma}
\begin{proof}
Using the commutation relations (\ref{comm}), 
$$ HXf=XHf+2Xf=mXf+2Xf=(m+2)Xf$$
which implies $Xf$ has right $K$-type $m+2$.  Also, 
$$ HYf=YHf-2Yf=mYf-2Yf=(m-2)Yf.$$
which implies $Yf$ has right $K$-type $m-2$.  
\end{proof}

Let $\mu$ be a Radon measure on the Borel subsets of $G$ that is non-zero on non-empty open sets and invariant under left (resp. right) translation, which exists by Theorem 2.2 of \cite{tou}.  We call this measure a \textit{left} (resp. \textit{right}) \textit{Haar measure} on $G$.  

Let $\Gamma$ be a discrete subgroup of $G$.  Because $G$ is a connected (semi)simple Lie group, and because $\Gamma$ is discrete, Proposition 3.6 of \cite{tou} tells us that $G$ and $\Gamma$ are \textit{unimodular} --- that is, a left-invariant Haar measure will also be right-invariant, and vice-versa.  (An example of a group that is \textit{not} unimodular is the subgroup of $SL(2,\RR)$ consisting of upper-triangular matrices.)
So, by Theorem 4.2 of \cite{tou}, there exists a Radon measure $\xi$ on $G / \Gamma$, which we will refer to as a Haar measure on $G / \Gamma$ by abuse of terminology.  It is unique up to multiplication by a strictly positive scalar, and, when normalized appropriately, satisfies
\begin{align}\label{quotint}
\int_G f(g) \mu (g) = \int_{G / \Gamma} \sum_{\gamma \in \Gamma} f(g \gamma) \xi(g \Gamma)
\qquad (f \in C_c(G))
\end{align}
In particular, this formula holds for $f \in L^2(G)$, since compactly-supported continuous functions are dense in $L^p(G)$ for $1 \leq p < \infty$.

We call $\Gamma$ a \textit{lattice} in $G$ if $G / \Gamma$ supports a finite Haar measure.

Suppose there exists a set $D$ among the Borel-measurable subsets of $G$ such that the canonical projection $G \rightarrow G / \Gamma$ restricted to $D$ is a bijection.  We call such a set $D$ a \textit{strict fundamental domain} for $G / \Gamma$ in $G$.  We call a set $F$ among the Borel measurable sets of $G$ which differs from a strict fundamental domain by a set of measure zero with respect to $\mu$ a  \textit{fundamental domain} for $G / \Gamma$ in $G$.  By Proposition 4.8 of \cite{tou}, we can always find a fundamental domain for our $G$ and $\Gamma$; and by Proposition 4.10 of \cite{tou},  we may integrate over $G / \Gamma$ by integrating over $F$, or over $D$:
\begin{align}\label{fdint}
\int_{G / \Gamma} f(g \Gamma) \xi(g \Gamma) = \int_F f(g \Gamma) \mu(g) =\int_D f(g \Gamma) \mu(g)
\qquad (f \in C_c(G / \Gamma))
\end{align}
Again, this formula holds for $f \in L^2(G/\Gamma)$, since compactly-supported continuous functions are dense in $L^p(G/\Gamma)$ for $1 \leq p < \infty$.

We conclude this chapter with some examples of lattices in $G$: \textit{Fuchsian groups} (discrete subgroups of $SL(2,\RR)$) \textit{of the first kind} (with finite covolume in $G$).  

Let $\Gamma$ be a Fuchsian group of the first kind.  An element $\gamma$ of $\Gamma$ (in fact, any element of $SL(2,\CC)$) can be classified by its trace: 

\begin{center}
    \begin{tabular}{| l | l |}
    \hline
    $\vert \text{tr} \gamma \vert < 2$  & $\gamma$ is \textit{elliptic}  \\ \hline
    $\vert \text{tr} \gamma \vert = 2$  & $\gamma$ is \textit{parabolic}  \\ \hline
    $\vert \text{tr} \gamma \vert < 2$  & $\gamma$ is \textit{hyperbolic}  \\ 
    \hline 
    \end{tabular} 
\end{center} 

Note that the action of $-I$ by linear fractional transformation on $\HHH$ is trivial.  Let $Z(\Gamma)$ denote the center of $\Gamma$.  Using the Borel density theorem, it can be seen that $Z(\Gamma)=\lbrace \pm I \rbrace \cap \Gamma $.  (The Borel density theorem states that $\Gamma$ is Zariski-dense in $G$.  The Zariski-closed sets in $G$ are solutions to polynomial equations in matrix entries.  For a fixed element in $\Gamma$, consider the conjugation map from $\Gamma$ to the conjugacy class of $h$ in $\Gamma$.  As a polynomial in matrix entries, this map is Zariski-continuous, hence extends to $G$.  If $h$ is in $Z(\Gamma)$, the conjugacy class of $h$ in $\Gamma$ is finite, hence the image of the conjugation map extended to $G$ is finite, and the set of all elements that commute with $h$ in $G$ is a closed subgroup of finite index in $G$.  But $G$ is Zariski-connected, so in fact the elements in $G$ commuting with $h$ form all of $G$, and we have $h\in Z(G)$.  This shows $Z(\Gamma) \subset Z(G)=\lbrace \pm I \rbrace$.)  Let $\iota : \Gamma \hookrightarrow \text{Aut}(\HHH)$ be the embedding of $\Gamma$ (acting via linear fractional transformations) into the automorphisms of $\HHH$.  Then $\iota(\Gamma) \cong \Gamma / Z(\Gamma)$.

A \textit{cusp} of $\Gamma$ is a point in $\RR \cup \left\lbrace \infty \right\rbrace$ whose stabilizer in $\Gamma$ contains a non-trivial unipotent matrix, e.g. $\infty$ is a cusp of $SL(2,\ZZ)$ because it is fixed by the parabolic element $ 
\begin{psmallmatrix} 
  1     & 1\\
  0 & 1 
\end{psmallmatrix}
$.  A fundamental domain for $\Gamma \backslash G$ is non-compact if and only if $\Gamma$ has at least one cusp. So, the space $\Gamma \backslash G / K = \Gamma \backslash \HHH$ is non-compact if and only if $\Gamma$ has at least one cusp.

Suppose $\Gamma$ has a cusp, and so $\Gamma \backslash \HHH$ is non-compact (but still of finite covolume, as $\Gamma$ is a lattice).  Then we can topologize $\Gamma \backslash \HHH$ in such a way that it becomes endowed with the structure of a compact Riemann surface of genus $g$.  (See \cite{miy} Section 1.7.)

A Fuchsian group of the first kind is completely determined by the genus of $\Gamma \backslash \HHH$, by the elliptic elements and their orders, and by the number of inequivalent cusps:

\begin{proposition} \textit{(Proposition 2.4 in \cite{iwan}, where it is stated for discrete subgroups of finite covolume in $PSL(2,\RR)$ rather than $SL(2,\RR)$; due to Fricke and Klein)} \label{fuchgen}
Let $\Gamma$ be a Fuchsian group of the first kind.  Then $\Gamma / Z(\Gamma)$ is generated by motions 
$$A_1, \ldots , A_g, B_1 , \ldots B_g , E_1 , \ldots , E_l, P_1,...P_h$$ 
satisfying the relations 
\begin{gather*}
\left[ A_1,B_1 \right] \ldots \left[ A_g,B_g \right] E_1 \ldots E_l P_1 \ldots P_h = 1, \\
E_j ^{m_j} = 1, \,\,\,\,\, 1 \leq j \leq l,
\end{gather*}
where $A_j,B_j$ are hyperbolic motions, $\left[ A_1,B_1 \right]$ is the commutator $A_jB_jA_j^{-1}B_j^{-1}$, $g$ is the genus of $\Gamma \backslash \HHH$, $E_j$ are elliptic motions of order $m_j \geq 2$, $P_j$ are parabolic motions, and $h$ is the number of inequivalent cusps.
\end{proposition}

The \textit{signature} of a Fuchsian group of the first kind is 
\begin{align*}
(g ; m_1, \ldots , m_l ; h).
\end{align*}
in the notation of Proposition \ref{fuchgen}.

Calculated with respect to the measure $y^{-2}dxdy$ on $\HHH$, the area of $\Gamma \backslash \HHH$ is given by the Gauss-Bonnet formula (pg. 33 of \cite{iwan}, where it is given in the context of subgroups of $PSL(2,\RR)$),
\begin{align} \label{gb}
\frac{\text{vol}( \Gamma \backslash \HHH )}{2 \pi} = 2g-2+\sum_{j=1}^l \left( 1-\frac{1}{m_j} \right) + h
\end{align}

Let $N.A.K$ be the Iwasawa decomposition of $G$ ($N$ consists of the upper-triangular matrices with $1$ on the diagonal, and $A$ consists of the diagonal matrices in $G$), and normalize the Haar measure $dn.da.dk$ so that $\int_K dk = 1$.  Then $\text{vol}(\Gamma \backslash \HHH) = \text{vol}(\Gamma \backslash G)$.

Let $\bar{G}=G / \lbrace \pm I \rbrace=PSL(2,\RR)$. Suppose $\Gamma$ contains $-I$, and let $\bar{\Gamma}= \Gamma / \lbrace \pm I \rbrace$.  Then $\text{vol}(\Gamma \backslash G) = \text{vol}(\bar{\Gamma} \backslash \bar{G})$.

Let 
\begin{align*}
H_q = \left\langle \begin{pmatrix}
1 & \lambda \\
0 & 1
\end{pmatrix},
\begin{pmatrix}
0 & -1 \\
1 & 0
\end{pmatrix}
\right\rangle 
/ \lbrace \pm I \rbrace,
\end{align*}
where
$$ \lambda = 2 \cos \left( \frac{q}{3} \right), \,\, q\in \ZZ , \,\, q \geq 3.$$
These groups, sometimes called ``Hecke groups,'' are discrete groups of finite covolume in $\bar{G}$, hence they are Fuchsian groups of the first kind.  From the analysis of $H_q$ in Appendix III in \cite{toa}, from Proposition \ref{fuchgen}, and from the Gauss-Bonnet formula (\ref{gb}), we can fill in the following table:

\begin{center}
    \begin{tabular}{| c | c | c | c |}
    \hline
    group      &signature     &isomorphic to    &covolume in $\bar{G}$ \\ \hline
    $H_q$      &$(0; 2,q; 1)$ &$Z_2 * Z_q$      &$\pi \left( 1 - \frac{2}{q} \right)$ \\ 
    \hline 
    \end{tabular} 
\end{center}
Note $H_3=PSL(2,\ZZ)$.  

There is a theory of ``principal congruence subgroups'' for each $H_q$, but only $H_3$ is well-understood.  (For progress on general $H_q$, see, for example, \cite{llt}.)  We now give examples of congruence subgroups of $H_3$.

Denote $SL(2,\ZZ)$ by $\Gamma(1)$.  (Note $\bar{\Gamma}(1) = H_3$).  The \textit{principal congruence subgroup} $\Gamma(N)$ is the kernel of the canonical map $\Gamma(1) \rightarrow SL(2, \ZZ / N \ZZ)$.  It has the form
\begin{align*}
\Gamma(N) = \left\lbrace \begin{pmatrix}
a & b \\
c & d
\end{pmatrix} \in \Gamma(1) \, : \, a \equiv d \equiv 1 (\text{mod }N), \, b \equiv c \equiv 0 (\text{mod }N) \right\rbrace.
\end{align*}

A \textit{congruence subgroup} is any subgroup of $\Gamma(1)$ containing a principal congruence subgroup.  The group \begin{align*}
\Gamma_0(N) = \left\lbrace \begin{pmatrix}
a & b \\
c & d
\end{pmatrix} \in \Gamma(1) \, : \, c \equiv 0 (\text{mod }N) \right\rbrace
\end{align*}
is a congruence subgroup containing the principal congruence subgroup $\Gamma(N)$.  These are sometimes called ``Hecke congruence subgroups of level $N$.''

The structure of many congruence subgroups can be read from the tables in \cite{cp}.  (We thank Martin Westerholt-Raum for pointing out these tables.)  If we are looking for congruence subgroups that are isomorphic to the free group on $n$ generators, then by Proposition \ref{fuchgen}, we will want groups with genus 0, no elliptic elements, and $n+1$ cusps; and we can calculate the covolumes of each lattice from formula (\ref{gb}).  For example, 

\begin{center} \label{Fn}
    \begin{tabular}{| c | c | c | c |}
    \hline
    group              &signature      &isomorphic to  &covolume in $\bar{G}$ \\ \hline
    $\bar{\Gamma}_0(4)$      &$(0; - ; 3)$ &$F_2$          &$2\pi$ \\  \hline
    $\bar{\Gamma}_0(4) \cap \bar{\Gamma}(2) $      &$(0; - ; 4)$ &$F_3$          &$4\pi$ \\ \hline
    $\bar{\Gamma}(4)$      &$(0; - ; 6)$ &$F_5$          &$8\pi$ \\ 
    \hline 
    \end{tabular} 
\end{center}

Notice that the covolumes scale with the index of the subgroup.  For free groups, this index is given by the Nielsen-Schreier formula:  $\left[ F_n : F_{1+e(n-1)} \right] = e$.

\chapter{Discrete series representations in $L^2(SL(2,\RR))$} \label{d}

\textit{In this chapter, to lighten the notation, we write $G$ for $G_\DD$, $K$ for $K_\DD$, $\lag$ for $\lag_\DD$, and $j$ for $j_\DD$.  We will state and prove results in the unit disc model, but all results of this chapter can be transferred back to the upper-half plane model using (\ref{HtoD}).}

We will define a discrete series representation $D_m$ of $G$ by first defining a discrete series representation $D_{m,K}$ to be a certain $(\lag,K)$-module, as in Chapter 1 of \cite{vog}. Then we will give a natural realization of $D_m$ in the right regular representation of $G$ in $L^2(G)$, sketched in Section 15.10 of \cite{baf}.  (For a more analytic realization of $D_m$, see Chapters 16 and 17 of \cite{rob}.  We chose to use the realization of $D_m$ obtained from the Harish-Chandra module $D_{m,K}$ because in Chapter \ref{dseries}, a certain intertwiner will be defined on $D_{m,K}$.)  

A \textit{$(\lag, K)$-module} $(\pi,V)$ 
\begin{itemize}
\item is simultaneously a Lie algebra representation of $\lag$ and a representation of $K$, both denoted by $\pi$ on the same complex vector space $V$, which
\item has subspaces $\lbrace V_m \subseteq V : m \in \ZZ \rbrace$ such that 
\begin{itemize}
\item $V=\oplus _{m \in \ZZ} V_m$ (algebraic direct sum), 
\item $\pi(K)V_m \subseteq V_m$, 
\item $\pi(K)v=\chi_m(k)v$ for $v \in V_n$, and
\item $\pi(iH)v=imv$ for $v \in V_m$.  
\end{itemize}
\end{itemize}
Observe that the functions in $V$ must be infinitely differentiable, or \textit{smooth}, since $V$ is a Lie algebra representation of $\lag$.  (In Chapter \ref{discF}, we will \textit{define} ``smooth representations'' of a group for which infinite differentiability does not make sense, and these smooth representation will be seen to mimic the desirable properties of a $(\lag,K)$-module.  In particular, we will complete a smooth representation to obtain a unitary representation of the group.)  We say the $(\lag,K)$-module $(\pi,V)$ is \textit{admissible} if $\text{dim}V_m < \infty$.  These definitions can be found in Definition 1.1.7 of \cite{vog}, though the definitions there are given for complexified Lie algebras.

The set of \textit{$K$-types} of $(\pi,V)$ is $\lbrace n \in \ZZ \, : \, V_n \neq 0 \rbrace$.  An integer $m \in \ZZ$ is called a \textit{lowest $K$-type} of $(\pi,V)$ if it is a $K$-type of $(\pi,V)$ and $\vert m \vert $ is minimal with respect to this property. 

Note that an admissible $(\lag,K)$-module is not required to be a representation of $G$.  Suppose instead we start with the assumption that $(\pi,\HH)$ is a representation of $G$ such that $\pi$ restricted to $K$ is unitary.  Then Lemma 1.1.3 in \cite{vog} says that there exists a unique collection of closed, mutually orthogonal subspaces $\lbrace \HH_m \subseteq \HH : m \in \ZZ \rbrace$ such that 
\begin{itemize}
\item $\HH=\mathbin{\hat\oplus} _{m \in \ZZ} \HH_m$ (Hilbert space completion of algebraic direct sum),
\item $\pi(K)\HH_m \subseteq \HH_m$, and
\item $\pi(K)f=\chi_m(k)f$ for $f \in \HH_m$.  
\end{itemize}
We call the representation $(\pi,\HH)$ of $G$ \textit{admissible} if $\text{dim} \HH_m < \infty$.  We define the space of \textit{$K$-finite vectors} to be $\HH_{K}=\oplus _{m \in \ZZ} \HH_m$ (Definition 1.1.4 in \cite{vog}); it is a dense subspace of $\HH$.  

Showing that $(\pi,\HH_K)$ is an irreducible $(\lag,K)$-module is equivalent to showing that $(\pi,\HH)$ is an irreducible representation of $G$: 

\begin{proposition} \label{lagbij} (Proposition 1.1.6 in \cite{vog})
Suppose $(\pi,\HH)$ is an admissible representation of $G$.  Fix $W \in \lag$ and $f \in \HH_{K}$.  Then 
$$ \lim_{t \rightarrow 0} \frac{\pi(e^{tW}) f - f}{t} = \pi(W)f$$
exists, and defines a Lie algebra representation of $\lag$ on $\HH_K$. There exists a bijection between the set of closed, $G$-invariant subspaces of $\HH$, and arbitrary $\lag$-invariant subspaces of $\HH_K$.  A closed subspace $\mathcal{S}$ corresponds to $\mathcal{S}_{K}=\mathcal{S} \cap \HH_{K}$, and $\mathcal{S}=\overline{\mathcal{S}_{K}}$.  In particular, $\HH$ is topologically irreducible if and only if $\HH_{K}$ is algebraically irreducible.
\end{proposition}

If $(\pi,\HH)$ is an admissible representation of $G$, then the $(\lag, K)$-module $(\pi,\HH_K)$ is called the \textit{Harish-Chandra module of $\pi$}.  We say that two admissible representations of $G$ are \textit{infinitesimally equivalent} if their Harish-Chandra modules are isomorphic.  

\begin{theorem} \label{Hkext} 
(Proposition 1.1.9 in \cite{vog}, due to Harish-Chandra) Any unitary irreducible representation of $G$ is admissible.  Two such representations are infinitesimally equivalent if and only if they are unitarily equivalent.
\end{theorem}

For $m \neq 0$, we define the \textit{discrete series representation with parameter $m$}, denoted $D_{m,K}$, to be the unique irreducible admisible $(\lag,K)$-module with lowest right $K$-type $m+\text{sgn}(m)$.  The eigenvalue of the Casimir operator on $D_{m,K}$ is completely determined by $m$.

\begin{lemma} \label{caseig} (Lemma 1.2.11 in \cite{vog}) 
For $v \in D_{m,K}$, $\pi(\CCC)v = \frac{1}{2} ( \vert m \vert -1 )^2$.
\end{lemma}

Now we will build $D_{m,K}$.  The closure of its right $G$-translates, which we denote $D_m$, will turn out to be an irreducible subrepresention of the right regular representation of $G$ in $L^2(G)$; these are the discrete series representations of $G$ described in Chapter 17 of \cite{rob}.  

Let $f$ be a function on the unit disc $\DD$ in the complex plane, and let $w \in \DD$.  For $m \in \ZZ$, define $\tilde{f}$ by 
\begin{align} \label{tildef}
\tilde{f}(x)=f(x^{-1} . 0 ) j(x^{-1},0)^{-m} \qquad (x \in G).
\end{align}
Let $x \in G$ act on $f$ by 
$$(x \circ f)(w) = f(x^{-1}.w)j(x^{-1},w)^{-m}$$.  

\begin{lemma} \label{m} (similar to 5.13 in \cite{baf})
\begin{itemize}
\item[(i)] $\tilde{f}$ is of left $K$-type $m$.
\item[(ii)] $\widetilde{x \circ f}=\rho(x) \tilde{f}$.
\end{itemize}
\end{lemma} 
\begin{proof} 
\begin{align*} 
\lambda (k) \tilde{f}(x) &=  \tilde{f}(k^{-1}x) = f(x^{-1}k.0)j(x^{-1}k,0)^{-m} \qquad & \text{by definition} \\
&= f(x^{-1}.0)j(x^{-1},0)^{-m}j(k,0)^{-m} & \text{by (\ref{KDcoc}), and $K$ fixes 0} \\
&= \tilde{f}(x)\chi_m(k) &\text{by definition and (\ref{KDchar})} 
\end{align*} 
which proves (i).  For (ii),
\begin{align*}
\widetilde{x \circ f}(y)&= \left( x \circ f (y^{-1}.0) \right) j(y^{-1},0)^{-m} \\
&= f(x^{-1}y^{-1}.0) j(x^{-1},y^{-1}.0)^{-m} j(y^{-1},0)^{-m}   \qquad & \text{by definition} \\
&= f(x^{-1}y^{-1}.0) j (x^{-1}y^{-1},0)^{-m} & \text{by (\ref{cocD})} \\
&= \tilde{f}(yx) = \rho (x) \tilde{f} (y).
\end{align*}
\end{proof} 

\begin{lemma} \label{m2n} (similar to 5.17 in \cite{baf}) 
If $f_n(w)=w^n$, $n \in \NNN$, $w \in \DD$, then $\tilde{f}$ (resp.\ $\tilde{\bar{f}}$) is of right $K$-type $-m-2n$ (resp.\ $m+2n$).  
\end{lemma}
\begin{proof}
We will prove this lemma just for $\tilde{f}$; the proof for $\tilde{\bar{f}}$ is similar.  For $k_\theta \in K$,
\begin{align} \label{2n}
k_\theta w = \frac{e^{i\theta}w+0}{0w+e^{-i\theta}} = e^{2i\theta}w = \chi_2(k_\theta)w.
\end{align}
So we have
\begin{align*}
\tilde{f_n}(xk_\theta) &= \left( k_\theta^{-1} \left( x^{-1}.0 \right) \right)^n j ( k^{-1}x^{-1},0)^{-m} \\
&= \chi_{2n}(k_\theta^{-1}) (x^{-1}.0)^{n} j ( k^{-1},x^{-1}.0)^{-m} j ( x^{-1},0)^{-m} \qquad &\text{by (\ref{2n}) and (\ref{KDcoc})} \\
&= \chi_{2n}(k_\theta^{-1}) (x^{-1}.0)^{n} \chi_m(k_\theta^{-1}) j ( x^{-1},0)^{-m} \qquad &\text{by (\ref{KDchar})} \\
&= \chi_{-2n-m}(k_\theta^{-1}) \tilde{f_n}(x) =\chi_{2n+m}(k_\theta)\tilde{f_n}(x).
\end{align*}
\end{proof}

\begin{lemma} (similar to 5.17 in \cite{baf})
Let $f$ be holomorphic (resp.\ antiholomorphic).  Then $\tilde{f}$ is right $K$-finite if and only if $f$ is a polynomial in $w$ (resp.\ $\bar{w}$).
\end{lemma}
\begin{proof}
We will prove this lemma just in the holomorphic case; the proof in the antiholomorphic case is similar.  Suppose $\tilde{f}$ is right $K$-finite, so that $\tilde{f}= \sum_{i=1}^N \tilde{f_i}$ with each $\tilde{f_i}$ of a certain right $K$-type, and $f_i$ is holomorphic.  Without loss of generality we may assume $\tilde{f}$ is of right $K$-type $q$.  Applying the definitions, we have
$$ \widetilde{k \circ f}(y)=(k \circ f)(y^{-1}.0)j(y^{-1},0)^{-m}=f(k^{-1}y^{-1}.0)j(k^{-1},y^{-1}.0)^{-m}j(y^{-1},0)^{-m}.$$
Also, by assumption and by Lemma \ref{m} (ii),
$$ \widetilde{k \circ f}(y) = \rho(k)\tilde{f} = \chi_q(k)\tilde{f}(y) =\chi_q(k)f(y^{-1}.0)j(y^{-1},0)^{-m}. $$
Equating the right-hand sides, canceling $j(y^{-1},0)^{-m}$, and setting $w=y^{-1}.0$ gives
$$ f(k^{-1}.w)j(k^{-1},w)^{-m} = \chi_q(k)f(w),$$
and by (\ref{KDchar}),
$$ f(k^{-1}.w)\chi_{-m}(k) = \chi_q(k)f(w). $$
Since $f$ is holomorphic, we may write $f(w)=\sum c_n w^n$, convergent, with $n \geq 0$. Then 
\begin{align*}
\chi_{-m}(k) \sum c_n (k^{-1}.w)^n &=\chi_q(k) \sum c_n w^n \\
\chi_{-m}(k) \sum c_n \chi_{-2n}(k) (w)^n &=\chi_q(k) \sum c_n w^n
\end{align*}
which implies $q=-m-2j$ for some $j \in \NNN$, and $c_n=0$ for all $n \neq j$.
\end{proof}


Let $f$ be a polynomial in $w$ and let $\bar{f}$ be a polynomial in $\bar{w}$. For $m \geq 0$, let $E_m$ be the vector subspace of $C^\infty (G)$ spanned by the $\tilde{f}$, and let $E_{-m}$ be the vector subspace of $C^\infty (G)$ spanned by the $\tilde{\bar{f}}$.  

\begin{lemma} (Lemma 4.5 in \cite{baf}) \label{sqint}
For $m \in \NNN$, the function $g \mapsto j(g,0)^{-m}$ is bounded on $G$ if $m \geq 0$, square-integrable on $G$ if $m \geq 2$, and integrable on $G$ if $m \geq 4$.  
\end{lemma}
(Here is where the material on discrete series representations in \cite{baf} looks most like the material in \cite{rob}: Lemma \ref{sqint} is proved using the invariant metric and volume elements on $\DD$.)

\begin{proposition} \label{Em}
For $m \in \ZZ$, $(\rho , E_{m})$ is an irreducible admissible $(\lag, K)$-module, and $(\rho , \overline{\rho (G).E_{m}})$ is an irreducible representation of $G$.  For $\vert m \vert \geq 1$, $(\rho,E_{m+\emph{sgn}(m)})$ is a realization of $D_{m,K}$. $(\rho, \overline{\rho(G).E_{m+\emph{sgn}(m)}})$ is what we denoted earlier by $D_m$: a unitary irreducible representation of $G$ equivalent to (in this case, equal to) a subrepresentation of the right regular representation on $L^2(G)$.  Additionally, if $\vert m \vert \geq 3$, then $D_m$ consists of integrable functions.
\end{proposition}
\begin{proof}
By construction,
$$E_m= \oplus _{n \in \NNN} E_{m,-m-\text{sgn}(m)2n}.$$ 
By Lemma \ref{m2n}, 
$$\rho(k)f=\chi_{-m-\text{sgn}(m)2n} f \qquad (k \in K, \, f \in E_{m,-m-\text{sgn}(m)2n})$$ 
and 
$$\rho(K)E_{m,-m-\text{sgn}(m)2n} \subseteq E_{m,-m-\text{sgn}(m)2n}.$$ 
By Lemma \ref{hfmf}, 
$$iHf=i(-m-\text{sgn}(m)2n)f \qquad (f \in E_{m,-m-\text{sgn}(m)2n}).$$  
So, $(\rho,E_{m})$ is a $(\lag, K)$-module.  

By construction, 
$$\text{dim}E_{m,-m-\text{sgn}(m)2n} =1 < \infty \qquad (n \in \NNN),$$
which means that $(\rho,E_{m})$ is admissible.  

Finally, Lemma \ref{ef} implies that for the basis elements $X$ and $Y$ of $\lag^{\CC}$, we have
$$X(E_{m,-m-\text{sgn}(m)2n})=E_{m,-m-\text{sgn}(m)2n+2} \,\,\, \text{  and  }$$ 
$$Y(E_{m,-m-\text{sgn}(m)2n})=E_{m,-m-\text{sgn}(m)2n-2},$$ 
and because $E_{m,-m-\text{sgn}(m)2n}$ is 1-dimensional, this shows that the representation of $\lag^\CC$ is irreducible, and also that the representation of $\lag$ is irreducible.  Therefore $(\rho,E_{m})$ is an irreducible admissible $(\lag, K)$-module; and by Proposition \ref{lagbij}, $(\rho,\overline{\rho(G).E_{m}})$ is an irreducible representation of $G$.  (Compare this to the approach in Chapter 17 in \cite{rob}, where irreducibility is proven using complex analysis.)  So, for $\vert m \vert \geq 1$, $(\rho,E_{m+\text{sgn}(m)})$ and $(\rho,\overline{\rho(G).E_{m+\text{sgn}(m)}})$ satisfy the definitions of $D_{m,K}$ and $D_m$ given above.

By Lemma \ref{sqint}, $E_{m+\text{sgn}(m)}$ consists of functions that are bounded on $G$, square-integrable on $G$ if $\vert m \vert \geq 1$, and integrable on $G$ if $\vert m \vert \geq 3$.  Therefore $(\rho,\overline{\rho(G).E_{m+\text{sgn}(m)}})$ is a subrepresentation of the right regular representation of $G$ on $L^2(G)$, consisting of integrable functions if $\vert m \vert \geq 3$.
\end{proof}

\begin{proposition} \label{archSchur} \textit{(part of Theorem 16.3 in \cite{rob})}
Let $(\pi,\HH)$ be a unitary irreducible representation of $G$ in the discrete series.  Then there exists a constant $0 < \emph{d}_\pi \in \RR$ such that 
\begin{align} \label{Schurs}
\int _G (u,\pi(g)v)\overline{(u',\pi(g)v')}\mu(g) = \emph{d}_{\pi}^{-1} (u,u')\overline{(v,v')} \qquad (u,v,u',v' \in \HH)
\end{align}
\end{proposition}

The constant $\text{d}_\pi$ is called the \textit{formal dimension} of $\pi$.  (We note that if $K$ is a compact group, so that by the Peter-Weyl Theorem, any irreducible representation $\pi$ of $K$ is finite-dimensional, then the dimension of the representation is equal to $\text{d}_\pi \cdot \text{vol}(K)$.  For example, if Haar measure on the circle equals $2\pi$, then the formal dimension of the circle's $1$-dimensional irreducible unitary representations will be $\frac{1}{2\pi}$.  Both $\text{d}_\pi$ and $\text{vol}(K)$ depend on the Haar measure, but the dependence cancels out in the product.)  

Compare Proposition \ref{Em} to the discrete series construction in \cite{rob}, which we summarize in the next proposition.

\begin{proposition} \label{anD}
\textit{ (Lemma 17.6 and Proposition 17.7 in \cite{rob}) }  Let $H_k =L^2( \DD ,\mu_k)$, where $\mu_k$ is the measure 
\begin{align*}
\mu_k=(1- \vert u + iv \vert ^2 )^{k-2}dudv=(1-r^2)^{k-2}rdrd\theta \qquad (u+iv=re^{i\theta}=w \in \DD).
\end{align*}
Let $H_k^{\emph{hol}}$ be the subspace of $H_k$ consisting of holomorphic functions.  Then $H_k^{\emph{hol}}$ is closed in $H_k$. Define an action of $G$ on $H_k^{\emph{hol}}$ by
\begin{align*}
\pi_k(x)f(w)=j(x^{\emph{tr}},w)^{-k}f(x^{\emph{tr}}.w),
\end{align*}
where $x^{\emph{tr}}$ denotes the transpose of $x$.  Then $(\pi_k,H_k^{\emph{hol}})$ is a unitary representation of $G$.  For $k \geq 2$, the unitary representations $(\pi_k,H_k^{\emph{hol}})$ are irreducible and in the discrete series of $G$, and the formal dimension of $\pi_k$ is
\begin{align} \label{formdim1}
\emph{d}_k = \frac{k-1}{\pi}
\end{align}
when calculated with respect to the Haar measure on $G$ normalized so that $\int_K dk =1$.
\end{proposition}

If we take the Haar measure used to calculate the volumes at the end of Chapter \ref{setup} (natural for $SL(2,\RR)$ acting on $\HHH$), instead of the Haar measure in Proposition \ref{anD} (natural for $SU(1,1)$ acting on $\DD$), the formal dimension $\text{d}_k$ becomes 
$$ \text{d}_k = \frac{k-1}{4\pi} $$
because the transformation $T: \HHH \rightarrow \DD$ in Chapter \ref{setup} sends the measure $y^{-2}dxdy$ to the measure $4(1- \vert u + iv \vert ^2 )^{-2}dudv$.  

By a ``holomorphic discrete series representation of $PSL(2,\RR)$,'' we mean a holomorphic discrete series representation of $SL(2,\RR)$ that factors through $PSL(2,\RR)$.  In such a representation, $-I$ must act as the identity.  Examining the action in Proposition \ref{anD} (or the action in Proposition \ref{Em}), we see that this happens only for even $k$ (for odd $m$).  (Note that the indices are off by one:  $m=k-1$.)  

To summarize,
\begin{proposition}  \label{psldisc}
$(\rho,D_m)$, for $m$ odd, $m \geq 1$, is an irreducible unitary representation of $PSL(2,\RR)$ that is a subrepresentation of $L^2(PSL(2,\RR))$.  It consists of square-integrable functions if $m \geq 1$, and integrable functions if $m \geq 3$.  The formal dimension $\emph{d}_m$ of $(\rho,D_m)$ is 
$$ \emph{d}_m = \frac{m}{4\pi}.$$
\end{proposition}  


\chapter{Automorphic forms} \label{autdef}

The material in this chapter comes from \cite{baf} and \cite{bum}.

Let $\Gamma$ be a lattice in $G$.  
\begin{definition} \label{repdef}
A smooth function $f : G \rightarrow \CC$ is an \textit{automorphic form} for $\Gamma$ if it satisfies the following conditions:
\begin{itemize}
\item[A1] $f(\gamma g ) = f(g) \qquad ( \gamma \in \Gamma, \,\, g \in G)$
\item[A2] $f$ is $K$-finite on the right
\item[A3] $f$ is $\ZZZ$-finite on the right
\item[A4] $f$ satisfies the \emph{condition of moderate growth}:  There exist constants $C$ and $N$ such that $\vert f (g) \vert < C \Vert g \Vert ^N$, where the ``height'' $\Vert g \Vert$ is the length of the vector $\left( g, \emph{det}(g^{-1}) \right)$ in the Euclidean space $M_2 \left( \RR \right) \oplus \RR = \RR^5$.
\end{itemize}
\end{definition}

(The condition of moderate growth ensures that the space of automorphic forms is invariant under differentiation by elements of $\lag$; see Theorem 3.2.1 in \cite{bum}.) Suppose $\Gamma$ has a cusp at $c$, and let $f$ be in $L^2(\Gamma \backslash G)$.  If $c$ is not $\infty$, we may choose $\xi$ in $SL(2,\RR)$ such that $\xi(\infty)=c$.  
Define $f'$ 
by $f'(g)=f(\xi g)$.  Then $f'$ is in $L^2(\Gamma' \backslash G)$, where $\Gamma' = \xi^{-1} \Gamma \xi$; and $\infty$ is a cusp of $\Gamma'$, so $\Gamma'$ contains an element of the form 
$
\begin{psmallmatrix} 
  1     & r\\ 
  0 & 1 
\end{psmallmatrix}
$.
We say $f$ is \textit{cuspidal at c} if 
$$\int_0^r f' \left( 
\begin{pmatrix} 
  1     & x\\  
  0 & 1 
\end{pmatrix}
g \right) dx =0$$
We say $f$ is \textit{cuspidal} if it is cuspidal at every cusp. 

Let $\Lc ^2(\Gamma \backslash G)$ denote the space of cuspidal elements of $L^2(\Gamma \backslash G)$. Note that elements of $\Lc ^2(\Gamma \backslash G)$ are not necessarily ``cusp \textit{forms},'' because cusp forms, by definition, must satisfy A1-A4 in addition to cuspidality.  But the reverse is true:
\begin{lemma} \label{cuspsq}
A cusp form for $\Gamma$ is square-integrable on $\Gamma \backslash G$.
\end{lemma}
\begin{proof}
By Corollary 7.9 in \cite{baf}, cusp forms are rapidly decreasing at the cusps, which means they are bounded on $\Gamma \backslash G$, hence belong to $\Lc ^2(\Gamma \backslash G)$.  
\end{proof}

In fact, more is true: 
\begin{theorem} \label{dspec}
(\cite{baf}, 13.4) Let $\Lc ^2(\Gamma \backslash G)_m$ denote the functions in $\Lc ^2(\Gamma \backslash G)$ of right $K$-type $m$.  The spectrum of $\CCC$ (where the action of $\CCC$ is defined in the usual way, corresponding to a left-invariant vector field) in $\Lc ^2(\Gamma \backslash G)_m$ is discrete, with finite multiplicities.  The space $\Lc ^2(\Gamma \backslash G)_m$ has a Hilbert space basis consisting of countably many cusp forms that are eigenfunctions of $\CCC$.  In particular, cusp forms are dense in $\Lc ^2(\Gamma \backslash G)_m$.
\end{theorem}

If $\Gamma$ is cocompact, which is equivalent to $\Gamma$ having no cusps, then by writing ``$\Lc^2 (\Gamma \backslash G)$,'' we mean all of $L^2( \Gamma \backslash G)$.  


\begin{theorem} \label{crep}
(\cite{baf}, 16.2) The space $\Lc ^2 (\Gamma \backslash G)$ decomposes into a Hilbert direct sum of closed irreducible $G$-invariant subspaces with finite multiplicities.
\end{theorem}


\begin{theorem} \label{poin}
(\cite{baf}, 6.1) Let $\varphi$ be a function on $G$ that is integrable and $\ZZZ$-finite. Define $P_{\varphi}$ by 
\begin{align} \label{bigpoindef}
 P_{\varphi}(x) = \sum_{\gamma \in \Gamma} \phi(\gamma x)
\end{align}
\begin{itemize}
\item[(i)] If $\varphi$ is $K$-finite on the right, then the series $P_{\varphi}$ converges absolutely and locally uniformly, belongs to $L^1(\Gamma \backslash G)$, and represents an automorphic form for $\Gamma$.
\item[(ii)] If $\varphi$ is $K$-finite on the left, then $P_{\varphi}$ converges absolutely and is bounded on $G$.
\end{itemize}
\end{theorem}

We call such a series in Theorem \ref{poin} a \textit{Poincar\'{e} series}.  Under additional assumptions on $\varphi$, Poincar\'{e} series turn out to be cusp forms:

\begin{theorem} \label{poinint}
(\cite{baf}, 8.9) Let $\varphi$ be a function on $G$ that is $\ZZZ$-finite, $K$-finite on both sides, and belongs to $L^1(G)$.  Then the Poincar\'{e} series $P_{\varphi,\Gamma}$ is a cusp form for $\Gamma$.  
\end{theorem} 

The first ``Poincar\'{e} series'' were defined on $\DD$ (or $\HHH$), not on $G$.  The series in the next theorem are \textit{classical Poincar\'{e} series}.  (These are not quite the same classical Poincar\'{e} series as those discussed in Chapter 3 of \cite{iwan}.  We will talk about the relationship between these two kinds of classical Poincar\'{e} series in the next chapter.)

\begin{theorem} (\cite{baf}, 6.2) \label{poind}
Let $m$ be an integer, $m \geq 4$.  Let $\xi$ be a bounded holomorphic function on $\DD$.  Then the series 
\begin{align} \label{poindef}
p_\xi ^m (w) = \sum _{\gamma \in \Gamma} j(\gamma,w)^{-m} \xi(\gamma . w)
\end{align}
converges absolutely and locally uniformly and defines a holomorphic automorphic form of weight $m$.  The function $g \mapsto j(g,0)^{-m} \xi(\gamma . 0)$ is in $L^1(\Gamma \backslash G)$ and is bounded if $\xi$ is a polynomial.
\end{theorem}  

These series are classical automorphic forms:  They satisfy conditions similar to A1-A4 in Definition \ref{repdef}, but on $\DD$ instead of $G$.  Moreover, they are cusp forms: they vanish at the cusps of $\Gamma$.

\begin{theorem} \label{poincusp}
The series $p_\xi ^m (w)$ defined in the previous theorem is a cusp form for $\Gamma$ of weight $m$. 
\end{theorem}
\begin{proof}
This is proven as part of Lemma 8.5 in \cite{kra}.
\end{proof}

For a Fuchsian group of the first kind, the cusp forms comprise a finite-dimensional complex vector space, with dimension given by:  

\begin{theorem} \label{dimcusp} \textit{ (\cite{miy}, Theorem 2.5.2) }
Let $m$ be an even integer, $\Gamma$ a Fuchsian group of the first kind, $g$ the genus of the compactification of $\Gamma \backslash G / \HHH$, $e_1, \ldots , e_r$ the orders of inequivalent elliptic points of $\Gamma$, $h$ the number of inequivalent cusps of $\Gamma$, and $S_m(\Gamma)$ the space of cusp forms of weight $m$ for $\Gamma$.  Then
\begin{align*} 
\emph{dim}S_m (\Gamma) = 
 \begin{cases} 
      (m-1)(g-1) + \sum_{i=1}^r \left\lfloor \frac{m}{2} \left( 1 - \frac{1}{e_i} \right) \right\rfloor + \left( \frac{m}{2} - 1 \right) t    & (m>2) \\ 
      g        & (m=2)  \\
      1        & (m=0, h=0)  \\
      0        & (m=0, h>0)  \\     
      0        & (m<0)  \\ 
   \end{cases}
\end{align*}
\end{theorem}

\chapter{Discrete series representations in $L^2(\Gamma \backslash SL(2,\RR))$} \label{dseries}

\textit{As in Chapter \ref{d}, to lighten the notation, we write $G$ for $G_\DD$, $K$ for $K_\DD$, $\lag$ for $\lag_\DD$, and $j$ for $j_\DD$.  We will state and prove results in the unit disc model, but all results of this chapter can be transferred back to the upper-half plane model using (\ref{HtoD}).}


Using $k$ and $m$ as in Chapter \ref{d}, with $m=k-1$:  Let $\text{mult}(k-1,\Gamma), \,\, k-1 \geq 3,$ denote the multiplicity of $(\rho, D_{k-1})$ in $L^2(\Gamma \backslash G)$.   The goal of this chapter and the next will be to explain the following diagram:  

\[
\begin{tikzcd}[column sep = large, row sep = large] \label{multdiagram}
P_{\tilde{f}_{n,k}}(g) \not\equiv 0 \text{ for some } n  \,\,\,\,  \arrow[r, Rightarrow, "\,\,\,\, \text{Proposition \ref{bigphi}} \,\,\,\, "] \arrow[d, Leftrightarrow, "\text{Conjecture \ref{lrpoin}}"]
& \,\,\,\, \text{mult}(k-1,\Gamma) \neq 0 \arrow[d, Leftrightarrow, "\text{Theorem \ref{multdim}}"] \\
p^k_{f_n}(w) \not\equiv 0 \text{ for some } n \,\,\,\,  \arrow[r, Rightarrow, " \,\,\,\, \text{Theorem \ref{poincusp}} \,\,\,\, "]
& \,\,\,\,  \text{dim}S_k(\Gamma) \neq 0
\end{tikzcd}
\]

The dimension of the space of cusp forms for $\Gamma$ is related to the multiplicity of $(\rho, D_{k-1})$ by:  

\begin{theorem} \label{multdim}
\textit{(\cite{gelb}, Theorem 2.10)}  Let $\Gamma$ be a Fuchsian group of the first kind, and let $k-1 \geq 1$.  The representation $(\rho,D_{k-1})$ of $G$ occurs in $ \Lc ^2 (\Gamma \backslash G)$ with multiplicity equal to $\emph{dim} S_{k}(\Gamma)$.  
\end{theorem}

Next, we will need a certain intertwiner.  


\begin{theorem} \label{bigphi}
The formation of Poincar\'{e} series provides an intertwiner $\bar{\Phi}$ from a holomorphic discrete series representation $D_{k-1} \subset L^2(G)$ into $\Lc ^2(\Gamma \backslash G)$.  This intertwiner is either zero, or it is a unitary equivalence between $D_{k-1}$ and a discrete series representation of the same lowest right $K$-type, $D_{k-1,\Gamma} \subset \, \Lc ^2 (\Gamma \backslash G)$.
\end{theorem}
\begin{proof}
For $k  \geq 4$, $D_{k-1,K}$ has a basis consisting of functions $\tilde{f}_{n,k}$ of left $K$-type $k$ and right $K$-type $-k-2n$, $n \in \NNN$; and these functions are integrable, by Proposition \ref{Em}.   By Lemma \ref{caseig}, $\tilde{f}_{n,k}$ is an eigenvector of $\CCC$; in particular, $f_{n,k}$ is $\ZZZ$-finite.  So $\tilde{f}_{n,k}$ satisfies the conditions of Theorem \ref{poinint}, and we have a linear map 
$$ \Phi \, : \, \tilde{f}_{n,k} \mapsto P_{\tilde{f}_{n,k}},$$ 
where $P_{\tilde{f}_{n,k}}$ is a cusp form for $\Gamma$.  By Lemma \ref{cuspsq}, $P_{\tilde{f}_{n,k}} \in \ \Lc ^2(\Gamma \backslash G)$, hence $P_{\tilde{f}_{n,k}}$ is in one of the irreducible unitary representations in Theorem \ref{crep}.

Suppose $\Phi$ is non-zero.  Then by Schur's Lemma, $\Phi$ is surjective onto a Harish-Chandra module of one of the irreducible unitary subrepresentations in $\Lc ^2 (\Gamma \backslash G)$ of Theorem \ref{crep}.  Now $\Phi$ is an infinitesimal equivalence defined on $D_{k-1,K}$, so by Theorem \ref{Hkext}, there exists a unitary equivalence $\bar{\Phi}$, which is a unitary intertwiner from $D_{k-1}$ onto $D_{k-1, \Gamma}$ satisfying
\begin{align*}
  \bar{\Phi}( \rho(g) f ) = \rho(g) \bar{\Phi}(f) \qquad (f \in D_{k-1}, \, g \in G).  
\end{align*}
\end{proof}

Note that there are no \textit{backwards} implications in the top and bottom rows of the diagram above.  There are two unanswered questions:
\begin{itemize}
\item[(1)] If we know a discrete representation occurs in $L^2(\Gamma \backslash G)$ for a given $k-1$, why must it have been constructed using the intertwiner $\bar{\Phi}$?
\item[(2)] If we know the dimension of the space of cusp forms for $\Gamma$ of weight $k$ is non-zero, why must at least one such cusp form come from the Poincar\'{e} construction applied to a monomial (or any other polynomial) on $\DD$?
\end{itemize}
We address these questions, and the left side of the diagram above, in the next chapter.

\chapter{Digression on Poincar\'e series} \label{secmetz}

We continue explaining the diagram in the previous chapter, using all the same notation.  This chapter is not essential to proving the two main results of the dissertation; it is merely a sidenote to the discussion of automorphic forms and discrete series representations.  

The following theorem gives a partial answer to Question 2 at the end of the previous chapter.  By ``Fuchsian group of the first kind acting on $\DD$,'' we mean the image of the Fuchsian group (which we defined to be in $SL(2,\RR)$) in $SU(1,1)$ after conjugation by the transformation $T$ from Chapter \ref{setup}.

\begin{theorem} \label{metz}
(\cite{metz}, Theorem 3) Let $\Gamma$ be a Fuchsian group of the first kind acting on $\DD$.  Assume that $0$ is not an elliptic point of $\Gamma$.  Then the construction of (classical) Poincar\'{e} series given by (\ref{poindef}) with $k=2q$ is an injective map from the space of polynomials on $\DD$ of degree $n \leq \emph{dim} S_{q-2}(\Gamma) - 2$ to the space of cusp forms.
\end{theorem}

(The key word above is ``injective.'')  A note on ``$k=2q$'':  In \cite{metz}, the Poincar\'{e} series are constructed with $\gamma ' (w)$ (the derivative of the action of $\gamma \in \Gamma$) in place of $j(\gamma, w)$ in (\ref{poindef}).  Writing out the linear fractional transformation and taking the derivative gives
$$ \gamma ' (w) = j(\gamma,w)^{-2}.$$

In particular, Theorem \ref{metz} says that for $k$ even, and $n\geq 0$, $p^k_{f_n}(w) \not\equiv 0 $ if $n \leq \text{dim} S_{\frac{k}{2}-2}(\Gamma) - 2$.  

This gives a backwards implication on the bottom row of our diagram, but only for cases when $\text{dim} S_{\frac{k}{2}-2}(\Gamma) - 2 \geq 0$.  Also, the requirement that 0 not be an elliptic point rules out the consideration of cases such as $\Gamma = T. SL(2,\ZZ). T^{-1}$. (Since $i$ is an elliptic point of $SL(2,\ZZ)$ and $T(i)=0$, $T. SL(2,\ZZ). T^{-1}$ has $0$ as an elliptic point.)  

We could obtain a corresponding backwards implication on the top row of the diagram, provided that the following conjecture is true:

\begin{conjecture} \label{lrpoin} 
$$\sum_{\gamma \in \Gamma} \tilde{f}_{n,k}(\gamma g) \neq 0 \,\,\, \text{ for some $g \in G$ } \,\,\,\, \Longleftrightarrow \,\,\,\,  \sum_{\gamma \in \Gamma} \tilde{f}_{n,k}(g \gamma) \neq 0 \,\,\, \text{ for some $g \in G$ }.$$
Or equivalently, writing $\tilde{f}_{n,k}$ explicitly as functions on $SU(1,1)$,
$$ \sum_{\gamma \in \Gamma} j((\gamma g)^{-1},0)^{-k} ((\gamma g)^{-1}.0)^n \not\equiv  0 \,\,\,\, \Longleftrightarrow \,\,\,\,  \sum_{\gamma \in \Gamma} j((g \gamma)^{-1},0)^{-k} ((g \gamma)^{-1}.0)^n \not\equiv  0 $$

\end{conjecture}
 
(Recall $P_{\tilde{f}_{n,k}}(g)= \sum_{\gamma \in \Gamma} \tilde{f}_{n,k}(\gamma g)$.)  If we can find just one $g$ where $\sum_{\gamma \in \Gamma} \tilde{f}_{n,k}(\gamma g) \neq 0$, why can't we do the same for $\sum_{\gamma \in \Gamma} \tilde{f}_{n,k}(g \gamma)$, and vice-versa?  (Clearly the conjecture is true if we know that one of the series does not vanish at $g=\mathbbm{1}$; for the two series are equal when $g=\mathbbm{1}$.)  We have tried using convolution, unimodularity of $G$, mutual transversals for left and right cosets of $\Gamma$, Dirac sequences, and even Moore's ergodic theorem, but we haven't been able to prove the conjecture.  

Here is how the if-and-only-if on the left side of the diagram would follow from this conjecture.   
\begin{align*}
j(g^{-1},0)^{-k}p_{f_n}^k(g^{-1}.0)&=j(g^{-1},0)^{-k} \sum_{\gamma \in \Gamma} j(\gamma^{-1},g^{-1}.0)^{-k} f_n(\gamma^{-1}.g^{-1}.0) &\text{by (\ref{poindef}) } \\
&= \sum_{\gamma \in \Gamma} j(\gamma^{-1},g^{-1}.0)^{-k} j(g^{-1},0)^{-k} f_n(\gamma^{-1}.g^{-1}.0) \\
&= \sum_{\gamma \in \Gamma} j(\gamma^{-1}g^{-1},0)^{-k} f_n(\gamma^{-1}.g^{-1}.0) &\text{by (\ref{cocD}) } \\
&= \sum_{\gamma \in \Gamma} j((g \gamma)^{-1},0)^{-k} f_n((g \gamma)^{-1}.0) \\
&= \sum_{\gamma \in \Gamma} \tilde{f}_{n,k}(g \gamma) &\text{by (\ref{tildef})}. \\
\end{align*}
And if the conjecture is true, we get the middle equality in
\begin{align*}
j(g^{-1},0)^{-k}p_{f_n}^k(g^{-1}.0) = 
\sum_{\gamma \in \Gamma} \tilde{f}_{n,k}(g \gamma) = \sum_{\gamma \in \Gamma} \tilde{f}_{n,k}(\gamma g) = 
P_{\tilde{f}_{n,k}}(g),
\end{align*}
thereby showing that 
\begin{align*}
p_{f_n}^k(g^{-1}.0) \neq 0 \,\, \text{ for some $w=g^{-1}.0$} \, \,\, \Longleftrightarrow \,\,\, P_{\tilde{f}_{n,k}}(g) \neq 0 \,\, \text{ for some $g \in G$ },
\end{align*}
which is the if-and-only-if on the left side of the diagram in the previous chapter.  

But without the conjecture, we are not sure how to prove the if-and-only-if on the left side of the diagram.

Note that the cocycle trick doesn't get us anywhere if we start with $P_{\tilde{f}_{n,k}}(g)$:
\begin{align*}
P_{\tilde{f}_{n,k}}(g) &= \sum_{\gamma \in G} \tilde{f}_{n,k} (\gamma g) \qquad &\text{by (\ref{bigpoindef})}\\
&= \sum_{\gamma \in G} j( (\gamma g ) ^{-1} , 0) ^{-k} f_n ( (\gamma g)^{-1}.0)  \qquad & \text{by (\ref{tildef}) } \\
&= \sum_{\gamma \in G} j( g^{-1}, \gamma ^{-1} . 0)^{-k} j(\gamma^{-1},0)^{-k} f_n ( g^{-1} \gamma^{-1} .0)  \qquad & \text{by (\ref{cocD}) } \\
\end{align*}


Now, \textit{suppose that Conjecture \ref{lrpoin} is true}, so that the if-and-only-if on the left side of the diagram holds.  Then the diagram (and the fact that we may change the basis of $D_{k-1,K}$ from a basis constructed by applying (\ref{tildef}) to monomials, to a basis constructed by applying (\ref{tildef}) to linearly independent polynomials) allows us to extend Theorem \ref{metz}:  By Schur's Lemma, we get

\begin{conjecture} \label{metzext}
(Corollary to Conjecture \ref{lrpoin}) Let $\Gamma$ be a Fuchsian group of the first kind acting on $\DD$, and suppose $0$ is not an elliptic point of $\Gamma$.  Let $k$ be an even integer such that $\emph{dim} S_{\frac{k}{2}-2}(\Gamma) - 2 \geq 0$, and let $\xi(w)$ be any polynomial on $\DD$.  Then the Poincar\'{e} series $p_\xi^k(w)$ defined by (\ref{poindef}) does not vanish identically.
\end{conjecture}

Here is how the series $p_{f_n}^k(w)$ are related to the series in Chapter 3 of \cite{iwan}.  (We thank Gergely Harcos for explaining this relationship.)

Suppose $\Gamma$ is a Fuchsian group of the first kind acting on $\DD$ (the context of Theorem \ref{poinint}, Theorem \ref{metz}, and Conjecture \ref{metzext}).  Transfer the the action back to $\HHH$ using $T$ of Chapter \ref{setup}, and \textit{assume $\Gamma$ contains $\Gamma_\infty$}, the stabilizer of the cusp at $\infty$.  Also, \textit{assume the function $\xi(w)$ is such that the Poincar\'{e} series in Theorem \ref{poinint} converges}, and let $f(z)$ be the function $\xi(w)$ transferred to the setting of $\HHH$ using $T$.  Break up the series as a double sum:

\begin{align} \label{poinchange}
\sum_{\gamma \in \Gamma} f(\gamma . z )j(\gamma, z)^{-k} =
 \sum_{\tau \in \Gamma_\infty \backslash \Gamma} \sum_{\rho \in \Gamma_{\infty} } j(\tau,z)^{-k}  f(\rho \tau . z) = 
 \sum_{\tau \in \Gamma_\infty \backslash \Gamma} j(\tau,z)^{-k} \sum_{\rho \in \Gamma_{\infty} } f(\rho \tau . z)
\end{align}
Define 
$$ g(u) = \sum _{\rho \in \Gamma_\infty} f(\rho . u).$$ 
Because $g(\rho_0 . u ) = g(u)$ for all $\rho_0 \in \Gamma_\infty$, $g(u)$ has a Fourier expansion
$$g(u) = \sum _{n=0} ^\infty c_n e(nz),$$
where $e(n z) = e^{2\pi i n z}$.  So from (\ref{poinchange}) we have 
\begin{align*}
\sum_{\gamma \in \Gamma} f(\gamma . z )j(\gamma, z)^{-k} =
\sum_{\tau \in \Gamma_\infty \backslash \Gamma} j(\tau,z)^{-k} \sum_{n=0}^\infty c_n e(n \tau . z)_ =
\sum_{n=0}^\infty c_n \sum_{\tau \in \Gamma_\infty \backslash \Gamma} j(\tau,z)^{-k} e(n \tau . z).
\end{align*}

The series 
\begin{align} \label{mainseries}
 \sum_{\tau \in \Gamma_\infty \backslash \Gamma} j(\tau,z)^{-k} e(n \tau . z)
\end{align}
are also called Poincar\'{e} series.  These are the kind of classical Poincar\'{e} series discussed in Chapter 3 of \cite{iwan}.  Open problems surrounding series such as (\ref{mainseries}) are:  What are the linear relations between them?  Which ones form a basis of $S_k(\Gamma)$?  Which ones are non-vanishing?  (See pg. 54 of \cite{iwan} for these problems.)  Perhaps Conjecture \ref{metzext} could be of use in answering these questions, if one had fine control over the Fourier coefficients $c_n$ above.  But it looks like there is no easy way to go back and forth between results on the two different types of classical Poincar\'{e} series.

\chapter{Von Neumann algebras}  \label{vna}

Most of the material in this chapter comes from \cite{toa}, \cite{js}, and \cite{vna}.

Let $\HH$ be any Hilbert space, and let $\MM$ be a subset of $\BH$, the bounded linear operators on $\HH$.  Let $\MM '$ denote the set of all operators in $\BH$ which commute with $\MM$, called the \textit{commutant} of $\MM$ on $\HH$, and let $\MM ''$ denote the set of operators in $\BH$ which commute with $\MM '$.  Note that $\MM '$ depends on the space $\HH$.  If the space is not obvious, so that $\MM'$ is ambiguous, then we will specify the space by writing $\text{End} _{\MM} \HH$ for the commutant of $\MM$ in $\BH$.

The \textit{strong operator topology} is the weakest topology on $\BH$ such that the maps 
$$E_x \, : \, \BH \rightarrow \HH, \qquad E_x(A)=Ax$$ 
are continuous for all $x \in \HH$.  
The \textit{weak operator topology} is the weakest topology on $\BH$ such that the maps 
$$E_{x,y} \, : \, \BH \rightarrow \CC, \qquad E_{x,y}(A)=(x,Ay)$$ 
are continuous for all $x,y \in \HH$.  

Let $\overline{\MM}^w$ denote the closure of $\MM$ in the weak operator topology, and let $\overline{\MM}^s$ denote the closure of $\MM$ in the strong operator topology.  

\begin{theorem} \label{bicom}
(von Neumann's bicommutant theorem; see e.g.\ Corollary 3.2.3 in \cite{vna})  If $\MM$ is a self-adjoint unital subalgebra of $\BH$, then $\MM '' = \overline{\MM}^w = \overline{\MM}^s$.
\end{theorem}

If $\MM$ is equal to $\MM '' = \overline{\MM}^w = \overline{\MM}^s$, we call $\MM$ a \textit{von Neumann algebra}.  

A \textit{factor} is a von Neumann algebra with trivial center.  A \textit{finite factor} is a factor that has a unique faithful normal (weakly and strongly continuous) tracial state.  The finite factors acting (irreducibly) on finite-dimensional complex vector spaces are the complex matrix algebras $M_n(\CC)$, called I$_n$ factors; the trace on projections in a I$_n$ factor attains all values in $\lbrace 0, \frac{1}{n}, \ldots , \frac{n}{n}= 1 \rbrace$ (normalized to equal $1$ on the identity).  Finite factors acting on infinite-dimensional Hilbert spaces are called II$_1$ factors; the trace on projections in a II$_1$ factor attains all values in $[0,1]$ (normalized to equal $1$ on the identity).  

Let $(\pi,\HH)$ and $(\pi', \HH')$ be two representations of any group $G$. We call a continuous (bounded) linear map $\sigma : \HH \rightarrow \HH'$ an \textit{intertwiner} if it commutes with $G$. That is, 
\begin{align*}
\sigma( \pi (g)  h ) = \pi ' (g)  \sigma (h) \qquad (h \in \HH, \, g \in G).
\end{align*}

\begin{lemma} \label{intvna}
Let $(\pi,\HH)$ and $(\pi', \HH')$ be two unitary representations of any group $G$. Suppose there exists a surjective intertwiner $\sigma : \HH \rightarrow \HH'$, and assume that $\overline{\pi(G)} ^s$ is a factor.  Then $\overline{\pi(G)} ^s$ and $\overline{\pi ' (G)} ^s$ are isomorphic as von Neumann algebras.  
\end{lemma}
\begin{proof}
We claim the map 
$$\pi(g) \,\,\,\, \mapsto \,\,\,\, \pi'(g)$$ 
is strong-strong continuous --- that is, for any net $\lbrace \pi(g_i) \rbrace_{i \in I}$,
$$ \Vert (\pi(g_i)-\pi(g))x \Vert _{\HH} \rightarrow 0 \qquad \Rightarrow \qquad
\Vert (\pi'(g_i)-\pi'(g))u \Vert _{\HH'} \rightarrow 0
$$
for all $x \in \HH$ and all $u \in \HH'$.

Given $u \in \HH '$, we can find $y \in \HH$ such that $u=\sigma (y)$, since $\sigma$ is surjective; and 
$$ \Vert (\pi'(g_i)-\pi'(g))u \Vert _{\HH '} = 
\Vert (\pi'(g_i)-\pi'(g)) \sigma(y) \Vert _{\HH '} =
\Vert \sigma ( (\pi(g_i)-\pi(g)) y ) \Vert _{\HH '} \leq
c \Vert (\pi(g_i)-\pi(g)) y \Vert _{\HH} 
$$
for some $c \geq 0$, since $\sigma$ is bounded.  This proves the claim.  So, the map $\pi(g) \mapsto \pi'(g)$ extends to a surjective $*$-homomorphism of von Neumann algebras, 
$$\overline{\pi(G)} ^s \,\,\,\, \rightarrow \,\,\,\, \overline{\pi ' (G)} ^s.$$
And because $\overline{\pi(G)} ^s$ is a simple ring (as it is a factor), this surjective homomorphism is an isomorphism.  
\end{proof}

Given a factor $\MM$ acting on a Hilbert space $\HH$, we can measure the ``size'' of $\HH$ as an $\MM$-module.  This ``size'' is given by the \textit{coupling constant of $\MM$ on $\HH$}, also called the \textit{von Neumann dimension of $\HH$ as an $\MM$-module}, denoted by $\text{dim} _{\MM} \HH$, which Murray and von Neumann define in Theorem X of \cite{roo1} to be 
$$\text{dim} _{\MM} \HH = \text{tr}'(P_{\overline{\MM v}}) / \text{tr}(P_{\overline{\MM' v}}), $$
and which has come to be defined as the reciprocal (pg. 3 in \cite{jon}),
\begin{align} \label{cc1}
\text{dim} _{\MM} \HH = \text{tr}(P_{\overline{\MM' v}}) / \text{tr}'(P_{\overline{\MM v}}), 
\end{align}
where $\MM'$ is the commutant of $\MM$ on $\HH$, $v$ is any non-zero vector in $\HH$, $P_{\overline{\MM' v}}$ and $P_{\overline{\MM v}}$ are projections onto the closures of the cyclic modules $\MM'v$ and $\MM  v$, and $\text{tr}$ and $\text{tr}'$ are the unique faithful normal tracial states on $\MM$ and $\MM'$. (Proving that this definition is independent of the choice of $v$ occupies several pages in \cite{roo1}.)  

Definition (\ref{cc1}) gives
$$ \text{dim} _{ M_n(\CC) } (\CC^n) = \frac{1}{n},$$
and the examples 
\begin{align} \label{ccmat}
\text{dim} _{ M_2(\CC) \otimes \mathbbm{1}_{\CC^3} } (\CC^2 \otimes \CC^3) &= \frac{3}{2}  \\
\text{dim} _{ M_2(\CC) \otimes \mathbbm{1}_{\CC^{30,000}} } (\CC^2 \otimes \CC^3) &= \frac{30,000}{2}
\end{align}
show that increasing the size of the representation space, thereby increasing the size of the commutant, is reflected by an increase in von Neumann dimension.  

In \cite{jon}, Vaughan Jones used the coupling constant to define the \textit{index} of one finite factor $\NN$ in another finite factor $\MM$, 
\begin{align*}
\left[ \MM : \NN \right] = \frac{\text{dim}_\HH \NN}{\text{dim}_\HH \MM},
\end{align*}
and showed that this ratio can only take values in 
\begin{align*}
\left\lbrace \cos \left( \frac{\pi}{q} \right), q=3,4,5... \right\rbrace \cup \left\lbrace r \geq 4, r \in \RR \right\rbrace  
\end{align*}

For example, $\left[ M_2(\CC) \otimes M_3(\CC) : M_2(\CC) \otimes \mathbbm{1}_{\CC^3} \right] = \frac{3/2}{1/6}=9$.  (We will do some simple subfactor index calculations involving II$_1$ factors, instead of I$_n$ factors, at the end of the next chapter.)

For a II$_1$ factor $\MM$, a definition of the coupling constant that is equivalent to (\ref{cc1}) and that better serves our purposes is
$$ \text{dim} _{\MM} \HH = \text{Tr}_{\MM'}(\mathbbm{1}_\HH),$$ 
where $\MM'$ is the commutant of $\MM$ on $\HH$, and $\text{Tr}_{\MM'}$ is the \textit{natural trace on $\MM '$}.  (The equivalence of this definition with (\ref{cc1}) is proved in Proposition 3.2.5(f) in \cite{toa}.)  We now explain what is meant by $\text{Tr}_{\MM'}$, which will involve defining two intermediate traces on commutants of $\MM$ taken in different spaces. 
 
Let $L^2(\MM)$ denote the Hilbert space obtained by completing $\MM$ with respect to the scalar product $(x,y)=\text{tr}(x^*y)$.  Lemma 3.2.2(a) in \cite{toa} says 
$$ \text{End}_\MM L^2(\MM) = J \MM J,$$ 
where $J$ is the conjugate-linear isometry extending the map $x \mapsto x^*$, $x \in \MM$.  Define 
$$\text{Tr}_{\text{End}_\MM L^2(\MM)}(JxJ)=\text{tr}_\MM(x) \qquad (x \in \MM).$$
Let $\KK$ be a Hilbert space with orthonormal basis $\lbrace \varepsilon_i \rbrace _{i \in I}$, and let $\MM$ act diagonally on $\KK \mathbin{\hat\otimes} L^2(\MM)$ (the Hilbert space completion of the algebraic tensor product) as $\mathbbm{1}_\KK \otimes \MM$.  Lemma 3.2.2(b) in \cite{toa} says 
$$\text{End}_\MM ( \KK \mathbin{\hat\otimes} L^2(\MM) )=\BK \otimes J\MM J.$$ 
Every $x \in \text{End} _\MM ( \KK \mathbin{\hat\otimes} L^2(\MM) )$ can be written as a ``matrix'' $\left( J x_{i,j} J \right) _{i,j \in I} $.  If $x$ is positive, then each diagonal entry $x_{i,i}$ is also positive, so we may define
$$ \text{Tr}_{\text{End}_\MM ( \KK \mathbin{\hat\otimes} L^2(\MM) )}(x)=\sum_{i \in I} \text{tr}(x_{i,i}) \qquad (x \in \text{End}_\MM ( \KK \mathbin{\hat\otimes} L^2(\MM) )_+ ),$$
and because every $x \in \text{End} _\MM (\KK \mathbin{\hat\otimes} L^2(\MM))$ can be written as a linear combination of at most four positive operators (Corollary 4.2.4 in \cite{kr1}), this definition extends to all of $\text{End} _\MM (\KK \mathbin{\hat\otimes} L^2(\MM))$.  Note that this trace, unlike the trace tr on the II$_1$ factor $\MM$, may be infinite:  For example, if $\KK$ is infinite-dimensional, let $p$ be a projection onto an infinite-dimensional subspace of $\KK$; then Tr$_{\MM'} ( p \otimes JxJ ) = \text{tr}_\MM(x) \text{dim}_\CC (p\KK) = \infty$.  

Assume now that $\KK$ is infinite-dimensional, and let $\HH$ be an $\MM$-module.  Lemma 3.2.2 (c) of \cite{toa} says that there exists an $\MM$-linear isometry 
$$ u \, : \, \HH \rightarrow \KK \mathbin{\hat\otimes} L^2(\MM). $$  
For $x \in \text{End}_\MM(\HH)$, we know $uxu^* \in   \text{End}_\MM \KK \mathbin{\hat\otimes} L^2(\MM) $, so we may define 
$$ \text{Tr}_{\text{End}_\MM(\HH)} (x) = \text{Tr} _{\text{End}_\MM ( \KK \mathbin{\hat\otimes} L^2(\MM) )} (uxu^*) \qquad (x \in \text{End}_\MM(\HH) _+),$$
and this definition extends to all of $\text{End}_\MM(\HH)$ (again, by decomposing operators into linear combinations of four positive operators).

So, for $\HH, \KK, \MM$, and $u$ as above, the coupling constant is 
$$\text{dim} _{\MM} \HH = \text{Tr}_{\text{End}_\MM(\HH)}(\mathbbm{1}_\HH) = \text{Tr}_{\text{End}_\MM ( \KK \mathbin{\hat\otimes} L^2(\MM) )}(uu^*).$$ 

Let $\Gamma$ be a discrete group, and let $\lambda _{\Gamma} (\Gamma)$ denote the left regular representation of $\Gamma$ on $l^2(\Gamma)$. $\lambda _{\Gamma} (\Gamma) ''$ is called the \textit{left group von Neumann algebra of} $\Gamma$, which we will denote by $L\Gamma$.  Let $\rho _{\Gamma} (\Gamma)$ denote the right regular representation of $\Gamma$ on $l^2(\Gamma)$.  $\rho _{\Gamma} (\Gamma) ''$ is called the \textit{right group von Neumann algebra of} $\Gamma$, which we will denote by $R\Gamma$.  The following proposition gives the criterion for such a von Neumann algebra to be a II$_1$ factor, and describes the relationship between $R\Gamma$ and $L\Gamma$. 

\begin{proposition} \label{iccfactor}
(similar to parts of Theorem 1.2.4 and Proposition 1.4.1 in \cite{js})  
Suppose $\Gamma$ has the property that every non-trivial conjugacy class has infinitely many elements.  Then $R\Gamma$ and $L\Gamma$ each have a unique faithful normal tracial state, hence they are II$_1$ factors. 

Let $J$ be the conjugate-linear isometry $\psi(\gamma) \mapsto \overline{\psi(\gamma^{-1})}$ on $l^2(\Gamma)$.  Then $J(R\Gamma)J=L\Gamma$, and $J(L\Gamma)J=R\Gamma$.  

Also, $\emph{End}_{R\Gamma} l^2(\Gamma) = L\Gamma$, and $\emph{End}_{L\Gamma} l^2(\Gamma) = R\Gamma$.
\end{proposition}

To take an example:  Let $\Gamma$ be a lattice in $PSL(2,\RR)$.  Then it follows from Borel's density theorem (which we mentioned in Chapter \ref{setup}) that every non-trivial conjugacy class in $\Gamma$ has infinitely many elements, and $R\Gamma$ is a II$_1$ factor (Lemma 3.3.1 in \cite{toa}).  

$R\Gamma$ and $L\Gamma$ are, of course, isomorphic as von Neumann algebras, as can be seen by letting $J$ be the surjective intertwiner in Lemma \ref{intvna}.  We made the distinction in order to provide a bit more intuition on the coupling constant. By Proposition \ref{iccfactor}, $R\Gamma$ and $L\Gamma$ are each other's commutants.  They are ``perfectly coupled'' on $l^2(\Gamma)$:    
$$\text{dim}_{R\Gamma} l^2(\Gamma) = 1 = \text{dim}_{L\Gamma} l^2(\Gamma).$$
Compare this to the coupling constant in (\ref{ccmat}).  

\begin{lemma} \label{unequiv}
Representations of finite factors are classified up to unitary equivalence by von Neumann dimension.
\end{lemma}
\begin{proof}
Let $\MM$ be a finite factor represented on a Hilbert space $\HH$.  Proposition 3.2.5 parts (e) and (i) in \cite{toa} state that
$$ \text{dim}_\MM (e \HH) = \text{tr} _{\MM'} (e) \text{dim}_\MM \HH$$
for $e$ be a projection in $\MM'$, and
$$ \text{dim}_\MM (\HH \otimes \KK) = \text{dim}_\CC(\KK) \cdot \text{dim}_{\MM} \HH.$$
This shows that we may obtain a representation space with any von Neumann dimension in $(0,\infty]$ using projections and amplifications.    Suppose we start with two representations of the $\MM$ having different von Neumann dimensions, and we apply the necessary projections and amplifications to obtain the standard representations (representations with von Neumann dimension equal to $1$).  By Theorem 7.2.9 of \cite{kr2}, two standard representations of two finite factors are unitarily equivalent if the two finite factors are algebraically isomorphic.  We were working with representations of one factor $\MM$, so we are done.  
\end{proof}

In other words, if there exists a unitary equivalence between two representations of a finite factor, then the von Neumann dimension of the factor is the same on both representations; or, if two representations of a finite factor have the same von Neumann dimension, then there exists a unitary equivalence between them.

\chapter{Von Neumann dimension of $D_m \subset L^2(PSL(2,\RR))$ as an $R\Gamma$-module} \label{discvnaextpf}

In this chapter, we state the result that Atiyah and Schmid proved in \cite{as}, in the form it appears in \cite{toa}.  We follow the proof in \cite{toa}, filling in some details along the way.  We conclude this chapter with examples in the case $G=PSL(2,\RR)$.  

\begin{theorem} \label{discvnaext}
(Theorem 3.3.2 in \cite{toa}.) Let $G$ be a connected real semi-simple non-compact Lie group without center. Let $\Gamma$ be a lattice in $G$. Then $R\Gamma$ is a II$_1$ factor. Let $(\pi, \HH)$ be a discrete series representation of $G$ (an irreducible unitary representation having square-integrable matrix coefficients) with $\HH \subset L^2(G)$.  Then the restriction of $\pi$ to $\Gamma$ extends to a representation of $R\Gamma$ on $\HH$, and
\begin{equation*} 
\emph{dim}_{R\Gamma}(\HH)=\emph{d}_\pi \cdot \emph{vol}(\Gamma \backslash G),
\end{equation*}
where $\emph{d}_\pi$ is the formal dimension of $\pi$.
\end{theorem}

\begin{proof}
Let $F$ be a fundamental domain for $G / \Gamma$ in $G$.  The proofs of Theorem 4.2 and Proposition 4.10 in \cite{tou} show that we may identify $L^2(F \times \Gamma)$ with $L^2(G)$.  

The following identification of Hilbert spaces is explained in Example 2.6.11 in \cite{kr1}.  Let $\varphi \in L^2(F)$, and $\psi \in l^2(\Gamma)$.  Linear combinations of functions of the form $f_{\varphi, \psi}(x, \gamma) =\varphi (x) \psi(\gamma)$ are dense in $L^2(F \otimes \Gamma)$ (since these functions include characteristic functions on finite-measure rectangles); and linear combinations of functions of the form $\varphi (x) \otimes \psi(\gamma)$ are dense in $L^2(F) \mathbin{\hat\otimes} l^2(\Gamma)$. We have a bounded linear map defined on the dense subspace
\begin{align} \label{denseid}
 \sigma \, : \,  \varphi (x) \otimes \psi(\gamma) \mapsto  f_{\varphi, \psi}(x, \gamma), 
 \end{align}
which extends by continuity to a surjective map
\begin{align} \label{idHilb}
  \sigma \, : \,  L^2(F) \mathbin{\hat\otimes} l^2(\Gamma) \rightarrow L^2(F \times \Gamma).
\end{align}

Let $\rho_{\Gamma}(\Gamma)$ denote the right regular representation of $\Gamma$ on $l^2(\Gamma)$, let $\rho_G(G)$ denote the right regular representation of $G$ on $L^2(G)$, and let $\rho_G(\Gamma)$ denote the restriction of $\rho_G(G)$ to $\Gamma$.  

First, note that
\begin{align} \label{firstiso}
 \overline{\mathbbm{1}_{L^2(F)} \otimes \rho_{\Gamma} (\Gamma)}^s = ( \mathbbm{1}_{L^2(F)} \otimes \rho_{\Gamma} (\Gamma) ) '' = \mathbbm{1}_{L^2(F)} \otimes \rho_{\Gamma} (\Gamma) '' \cong \rho_{\Gamma}(\Gamma)''=R\Gamma,
\end{align}
which is a factor, by Proposition \ref{iccfactor}.  Next, since $\psi$ is of the form $\sum _{\gamma' \in \Gamma} c_{\psi,\gamma'} \delta_{\gamma'}$, with $c_{\psi,\gamma'} \in \CC$, we may write 
$$f_{\varphi, \psi}(x, \gamma)=\varphi(x)\psi(\gamma)=\sum _{\gamma' \in \Gamma} \varphi(x) c_{\psi,\gamma'} \delta_{\gamma'} (\gamma) = c_{\psi,\gamma} f(x),$$  
and
$$\rho(\gamma_0) f_{\varphi, \psi}(x, \gamma) = 
\rho(\gamma_0) (c_{\psi,\gamma} \varphi(x)) = 
c_{\psi, \gamma \gamma_0} \varphi(x)=
( \rho(\gamma_0) \psi(\gamma) ) \varphi(x).$$
For the representations $\rho_G(\Gamma)$ on $L^2(G)=L^2(F \times \Gamma)$, and $\mathbbm{1}_{L^2(F)} \otimes \rho_{\Gamma}(\Gamma)$ on $L^2(F) \mathbin{\hat\otimes} l^2(\Gamma)$, we have  
$$ \sigma ( (\mathbbm{1}_{L^2(F)} \otimes \rho(\gamma_0) ) ( \phi(x) \otimes \psi(\gamma) ) )= \sigma ( \phi (x) \otimes \psi( \gamma \gamma_0 ) ) = \phi (x) \otimes \psi( \gamma \gamma_0 ) = \rho(\gamma_0) f_{\phi,\psi}(x,\gamma).$$
Thus $\sigma$ is a surjective intertwiner.  Combining this with (\ref{firstiso}), Lemma \ref{intvna} gives 
\begin{align} \label{seciso}
R\Gamma = \overline{\mathbbm{1}_{L^2(F)} \otimes \rho_{\Gamma} (\Gamma)}^s \cong \overline{\rho_G (\Gamma)}^s.
\end{align}

Let $p$ denote the projection onto the discrete series representation, 
$$p \, : \, L^2(G)\rightarrow \HH.$$  
Let $\pi(G)$ denote the representation of $G$ on $\HH$, and let $\pi(\Gamma)$ denote its restriction to $\Gamma$.  $\HH$ is invariant under $\pi(G)$, hence invariant under $\pi(\Gamma)$, so we have 
$$  p ( \rho(\gamma) f ) = \rho(\gamma) p(f) \qquad (f \in L^2(G), \, \gamma \in \Gamma).$$
So, $p$ is a surjective intertwiner from $L^2(G)$ to $\HH$, carrying the action of $\rho_G(\Gamma)$ onto the action of $\pi(\Gamma)$.  Therefore by (\ref{seciso}) and Lemma \ref{intvna},
$$ R\Gamma \cong \overline{\rho_G(\Gamma)}^s \cong \overline{\pi(\Gamma)}^s.$$ 
This shows that that $\HH$ is a $R\Gamma$-module, and that the discrete series representation of $G$ restricted to $\Gamma$ extends to a representation of $R\Gamma$.

Let us now compute the coupling constant.  (This part of the proof is carried out on pp. 145-147 of \cite{toa}.)  

$\HH$ is a closed subspace of $L^2(G)$.  $\HH$ is included as a $R\Gamma$-module in $L^2(G)$; let $u$ denote this inclusion, so that
$$p = uu^* : L^2(G) \rightarrow \HH.$$  
Using the definitions in Chapter \ref{vna}, 
$$ \text{dim}_{R\Gamma}\HH = 
\text{Tr}_{\text{End}_{R\Gamma} \HH }( \mathbbm{1}_{\HH} ) = 
\text{Tr}_{\text{End}_{R\Gamma} L^2(G)}(p).$$ 

By (\ref{idHilb}) and Proposition \ref{iccfactor}, we have 
$$\text{End}_{R\Gamma}L^2(G)=
\text{End}_{R\Gamma} ( L^2(F) \mathbin{\hat\otimes} l^2(\Gamma) ) = 
\mathcal{B}(L^2(F)) \otimes J R\Gamma J = 
\mathcal{B}(L^2(F)) \otimes L\Gamma
$$ 
generated by finite sums of the form 
\begin{align*} 
x = \sum_{\gamma \in \Gamma} a_\gamma \otimes \lambda(\gamma) \qquad (\lambda(\gamma) = J \rho(\gamma) J, \,\, \,\, a_\gamma \text{ a finite-rank operator on }\mathcal{B}(L^2(F)). 
\end{align*}
Let $ \lbrace \varepsilon_i \rbrace _{i \in \NNN}$ be an orthonormal basis for $L^2(F)$.  Define 
$$\bar{\varepsilon_i}\otimes \varepsilon_j (\phi) = (\varepsilon_i, \phi ) \varepsilon_j \qquad (\phi \in L^2(F))$$
so that we may write 
$$ a_\gamma = \sum _{i,j \in \NNN} a_{\gamma,i,j} \bar{\varepsilon_i} \otimes \varepsilon_j \qquad (a_{\gamma,i,j} \in \CC).$$
Every element of $\mathcal{B}(L^2(G))$ or $\mathcal{B}(L^2(F))$ is a linear combination of at most four elements of the positive cones $\mathcal{B}(L^2(G))_+$ or $\mathcal{B}(L^2(F))_+$, by Corollary 4.2.4 in \cite{kr1}; so, we can define a trace on $\mathcal{B}(L^2(G))$ or $\mathcal{B}(L^2(F))$ by defining the trace on $\mathcal{B}(L^2(G))_+$ or $\mathcal{B}(L^2(F))_+$.  First, define T$_{\mathcal{B}(L^2(F))}$ to be the trace on $\mathcal{B}(L^2(F))$ normalized so that T$_{\mathcal{B}(L^2(F))}(\bar{\varepsilon_i}\otimes \varepsilon_i ) = 1$ for all $i \in \NNN$. By definition,
$$ \text{Tr} _{\text{End}_{R\Gamma}L^2(G)}(a_\gamma \otimes \lambda(\gamma) ) = \text{tr}_{R\Gamma}(\lambda(\gamma)) \sum_{i \in \NN} a_{\gamma,i,i} = 
 \begin{cases} 
      0 & \gamma \neq \mathbbm{1} \\
      \text{T}_{\mathcal{B}(L^2(F))}\gamma & \gamma=\mathbbm{1}  
   \end{cases}
$$
where tr$_{R\Gamma}$ is the normalized trace on $R\Gamma$. Then
$$ \text{Tr} _{\text{End}_{R\Gamma}L^2(G)} (x)=\text{T}_{\mathcal{B}(L^2(F))}(a_{\mathbbm{1}}). $$
Next, define $T_{\mathcal{B}(L^2(G))}$ to be the trace on $\mathcal{B}(L^2(G))$ that is $1$ on rank-$1$ projections.  
Let $q$ be the orthogonal projection 
$$ q \, : \, L^2(G) \rightarrow L^2(F), \qquad f(g) \mapsto 
 \begin{cases} 
      f(g) & g \in F \\ 
      0 & g \not\in F  
   \end{cases}.
$$
Then
$$ \text{T}_{\mathcal{B}(L^2(F))}(y)=\text{T}_{\mathcal{B}(L^2(G))}(qyq) \qquad (y \in \mathcal{B}(L^2(G))_+, \text{ or } y \text{ a finite-rank operator on } L^2(G)),$$
and
$$ \text{Tr} _{\text{End}_{R\Gamma}L^2(G)} (x)=\emph{T}_{\mathcal{B}(L^2(F))}(a_{\mathbbm{1}})=\text{T}_{\mathcal{B}(L^2(G))}(q a_{\mathbbm{1}} q)=\emph{T}_{\mathcal{B}(L^2(G))}(qxq).$$
Every element of $\text{End}_{R\Gamma}L^2(G)_+$ is the strong limit of an increasing net of operators of the same form as $x$.  Because the trace is strongly continuous, the formula above holds for all elements of $\text{End}_{R\Gamma}L^2(G)_+$, and we have 
$$ \text{dim}_{R\Gamma}\HH = \text{Tr}_{\text{End}_{R\Gamma}L^2(G)}(p) = \text{T}_{\mathcal{B}(L^2(G))}(qpq) 
= \sum_{i \in \NNN} (qpq \varepsilon_i,\varepsilon_i) = \sum_{i \in \NNN} \Vert p \varepsilon_i \Vert ^2. $$
By (\ref{denseid}), we may view the orthonormal basis $ \lbrace \varepsilon_n \otimes \delta_\gamma \rbrace _{i \in \NNN, \gamma \in \Gamma}$ for $L^2(F) \mathbin{\hat\otimes} l^2(\Gamma)$ as an orthonormal basis $ \lbrace \rho(\gamma) \varepsilon_i \rbrace _{i \in \NNN, \gamma \in \Gamma}$ for $L^2(G)$. Let $\eta$ be a unit vector in $L^2(G)$, and assume $\eta \in \HH$, so that $p\eta=\eta$.  Then
$$ 1=\Vert \rho(g) \eta \Vert ^2 = \sum_{\gamma \in \Gamma} \sum_{i \in \NNN} \vert (\rho(g) \eta, \rho(\gamma) \varepsilon_i ) \vert ^2 \qquad (g \in G) $$
Also, 
$$\text{vol}( G / \Gamma)=\int_{G / \Gamma} \xi(g\Gamma)=\int_{F} \mu(g) \qquad \text{by (\ref{quotint}) and (\ref{fdint})},$$ 
and because $p \rho(g) = \rho(g) p$ on $\HH$, last integral may be written
$$ \sum_{i \in \NNN} \sum_{\gamma \in \Gamma} \int_F \vert (\rho(\gamma^{-1} g) p \eta,  \varepsilon_i ) \vert ^2 \mu(g) 
= \sum_{i \in \NNN} \int_G \vert (p \rho(g) \eta, \varepsilon_i) \vert ^2 \mu(g) 
= \sum_{i \in \NNN} \int_G \vert (\rho(g) \eta, p \varepsilon_i) \vert ^2 \mu(g).$$

Finally, by (\ref{Schurs}),
\begin{align} \label{formula}
\text{vol}(G / \Gamma) = \sum_{i \in \NNN} \text{d}_m^{-1} \Vert \eta \Vert ^2 \Vert p \varepsilon_i \Vert ^2 = \text{d}_m^{-1} \text{dim}_{R\Gamma}\HH 
\end{align}
Thus  Theorem \ref{discvnaext} is proved.  

\end{proof}



Here we give examples of the coupling constant calculated from Theorem \ref{discvnaext}, taking the ``$G$'' in the theorem to be our  $\bar{G}=PSL(2,\RR)$, and ``$\Gamma$'' in the theorem to be what we denoted earlier by $\bar{\Gamma}$, the image of a Fuchsian group of the first kind containing $\pm I$ in $\bar{G}$.  

By formula (\ref{gb}), and by Proposition \ref{psldisc}, we have for $m$ odd and $m \geq 1$,
\begin{align} \label{discform}
\text{dim}_{R\Gamma}D_m &= \frac{m}{4\pi} \cdot 2\pi\left( 2g-2+\sum_{j=1}^l \left( 1-\frac{1}{m_j} \right) + h \right) \\
&= \frac{m}{2} \left( 2g-2+\sum_{j=1}^l \left( 1-\frac{1}{m_j} \right) + h \right).
\end{align}

(This is a little more than the example given in \cite{toa}, which only uses the groups $H_q$ described in the second-to-last table of Chapter \ref{setup}.  Also, in the example in \cite{toa}, the authors should have restricted their $k$ to be even, corresponding to $m$ odd here, since only then do discrete series representations of $SL(2,\RR)$ factor through $PSL(2,\RR)$.)

Now, consider the free group examples in the last table in Chapter \ref{setup}.  The formula for coupling constants gives
\begin{align*}
\text{dim}_{R\bar{\Gamma}_0(4)}D_m &= m/2 \\
\text{dim}_{R ( \bar{\Gamma}_0(4) \cap \bar{\Gamma}(2) ) }D_m &= m \\
\text{dim}_{R\bar{\Gamma}(4)}D_m &= 2m 
\end{align*}

We can use this formula to calculate the index of some subfactors. (But note that using Theorem \ref{discvnaext} to calculate the indices of subfactors is overkill:  Covolume scales with the index of the subgroup, and the index of a subfactor here is really just the index of a subgroup, given by the Nielsen-Schreier formula, which was stated in the last sentence of Chapter \ref{setup}.)  Since
\begin{align*}
\bar{\Gamma}(4) \,\, \subset \,\, \bar{\Gamma}_0(4) \cap \bar{\Gamma}(2) \,\, \subset \,\, \bar{\Gamma}_0 (4),
\end{align*}
we have, taking weak closures in $D_m$,
\begin{align*}
R\bar{\Gamma}(4) \,\, \subset \,\, R \left( \bar{\Gamma}_0(4) \cap \bar{\Gamma}(2) \right) \,\, \subset \,\, R \bar{\Gamma}_0 (4), 
\end{align*}
or 
\begin{align*}
RF_5 \subset RF_3 \subset RF_2,
\end{align*}
so
\begin{align*}
\left[ RF_2:RF_3 \right] = 2 \\
\left[ RF_3:RF_5\right] = 2 \\
\left[ RF_2:RF_5 \right] = 4
\end{align*}
and the last line agrees with the result on the index of free group factors on pg. 348 of \cite{rad},
\begin{align*}
RF_N \cong M_k(\CC) \otimes RF_{(N-1)k^2+1}, \qquad [RF_N : RF_{(N-1)k^2+1} ] = k^2.
\end{align*}
(The isomorphism implies the index, but not the other way around.)

What if we want to consider a Fuchsian group of the first kind in $SL(2,\RR)$ that does \textit{not} contain $-I$  --- a lattice in $SL(2,\RR)$, not $PSL(2,\RR)$? Provided that it is an ICC group, the proof of Theorem \ref{discvnaext} carries over; and in formula (\ref{discform}), we can take any $m \geq 1$, not just the odd ones.

If we require that the lattice in $SL(2,\RR)$ is cocompact and torsion-free (so that we only have $g$ in the volume formula), and if we require $m \geq 3$, so that $D_m$ is integrable, then the coupling constant formula produces an integer.  This agrees with the following result of Langlands, summarized in Section 4 of \cite{bh} as:  
\begin{theorem} \label{dpigamma} 
(\cite{lan}) For a cocompact torsion-free lattice in a semisimple group $G$ over $\RR$, and an integrable irreducible unitary representation $\pi$ of $G$ in $L^2(\Gamma \backslash G)$,  
$$  \emph{d}_{\pi} \cdot \emph{vol}(G / \Gamma)  = d$$
where $\emph{d}_{\pi}$ is the formal dimension of $\pi$, and $d$ is the multiplicity of the representation.  
\end{theorem}

\chapter{Representing $R\Gamma_2$ on $D_{m,\Gamma_1} \subset L^2(\Gamma_1 \backslash PSL(2,\RR))$} \label{vnautrep}

\begin{theorem} \label{vnaautrep}
Let $G=PSL(2,\RR)$, and let $\Gamma_1$ and $\Gamma_2$ be lattices in $G$.  Then there exists a discrete series representation $D_{m,\Gamma_1} \subset L^2 (\Gamma_1 \backslash G)$ such that $W^*(\Gamma_2)$ has a representation on $D_{m,\Gamma_1}$, and this representation is unitarily equivalent to the representation of $W^*(\Gamma_2)$ on $D_{m} \subset L^2(G)$ given in Theorem \ref{discvnaext}, with the same von Neumann dimension.  
\end{theorem}
\begin{proof}
By Proposition \ref{psldisc}, we must restrict $m$ to be odd (to ensure that the representation of $SL(2,\RR)$ factors through $PSL(2,\RR)$).  Keeping this restriction in mind, Theorem \ref{multdim} shows that if $(\rho,D_{m})$, $m \geq 1$, does not occur in $L^2(\Gamma_1 \backslash G)$, then we may find a situation where a discrete series representation \textit{does} occur by raising the dimension of $S_{m+1}(\Gamma_1)$ above $0$, which can be done by replacing $m$ with a sufficiently large integer in Theorem \ref{dimcusp}.

The unitary equivalence of $D_m$ and $D_{m,\Gamma_1}$ as representations of $G$ gives a unitary equivalence of representations of $\Gamma_2$, which extends by Lemma \ref{intvna} to a unitary equivalence of representations of $W^*(\Gamma_2)$.  Lemma \ref{unequiv} says that the von Neumann dimension of $W^*(\Gamma_2)$ must be the same on both $D_m$ and $D_{m,\Gamma_1}$.
\end{proof}

The von Neumann dimension on $D_m$ is already given by formula (\ref{discform}).  So, provided that $(\rho, D_m)$ occurs in $L^2(\Gamma_1 \backslash G)$, the resulting coupling constant $\text{dim} _{R\Gamma_2} D_{m,\Gamma_1}$ does not depend on $\Gamma_1$.  

Let $\alpha$ be an element of a basis for $L^2(G / \Gamma_2)$, and let $\beta$ be an element of a basis for $l^2(\Gamma_1)$.  The following diagram illustrates the situations in Theorem \ref{discvnaext} and Theorem \ref{vnaautrep}:

\[
\begin{tikzcd}
l^2(\Gamma_2)  \arrow[d, equal] \\
\CC \alpha \otimes l^2(\Gamma_2)  \arrow[d, hookrightarrow] \\
L^2(G / \Gamma_2) \otimes l^2(\Gamma_2) \arrow[r, equal] & L^2(G)  \arrow[r, equal] & l^2(\Gamma_1) \otimes L^2(\Gamma_1 \backslash G) \\
& D_m  \arrow[u, hookrightarrow] \arrow[ddr, rightarrow, "\bar{\Phi}"', bend right=20] & l^2(\Gamma_1) \otimes  D_{m, \Gamma_1} \arrow[u, hookrightarrow] \\
& & \CC \beta \otimes  D_{m, \Gamma_1} \arrow[u, hookrightarrow] \\
& & D_{m, \Gamma_1} \arrow[u, equal] \\
\end{tikzcd}
\]

The left-hand side of the diagram illustrates the setting for the proof of Theorem \ref{discvnaext} in \cite{toa}.  The right-hand side of the diagram describes the setting for Theorem \ref{vnaautrep}.  To summarize parts of the diagram:  $D_{m,\Gamma_1}$ is a subrepresentation of the right regular representation of $G$ on $L^2(\Gamma_1 \backslash G)$, and $D_m$ is a subrepresentation of the right regular representation of $G$ on $L^2(G)$, but $D_{m,\Gamma_1}$ is \textit{not} a subrepresentation of the right regular representation of $G$ on $L^2(G)$.   Note that the first equality on the horizontal line is $\Gamma_2$-equivariant, and the second equality on the horizontal line is $\Gamma_1$-equivariant, but neither equality is $G$-equivariant.  

It might be interesting to see how this could be connected to the group measure space construction and questions in ergodic theory.  A related fact, proven in \cite{rie}:
\begin{align*}
\lbrace L^\infty(\Gamma_1 \backslash G ), \,\rho(\Gamma_2) \rbrace '' = \lbrace \lambda(\Gamma_1), \, L^\infty(G / \Gamma_2) \rbrace ',
\end{align*}
where the commutants are taken in $L^2(G)$.

\chapter{Non-archimedean local fields} \label{lnaf}

Let $\FF_q$ denote the finite field on $q$ elements, where $q$ is a power of some prime number $p \neq 2$.  The unique quadratic extension of $\FF_q$ is the finite field on $q^2$ elements, $\FF_{q^2}$.  (For example, if we start with the field $\FF_3 \cong \ZZ / 3\ZZ$, two quadratic extensions are given by $\FF_3 [x] / (x^2+1)$ and $\FF_3 [x] / (x^2+x+2)$, and both are isomorphic to $\FF_9$.  Another incarnation of $\FF_9$ is $\ZZ [\sqrt{2}]/(3)$.  These examples and more can be found in \cite{kcff}.)  

$\FF_{q^2}$ is a Galois extension of $\FF_q$, with Galois group isomorphic to $\ZZ / 2\ZZ$, generated by the map $x \mapsto x^q$, the non-trivial $\FF_q$-automorphism of $\FF_{q^2}$.  The \textit{norm} of an element $x \in \FF_{q^2}$ is 
$$\text{N}x=x \cdot \bar{x}=x \cdot x^q = x^{q+1}.$$  
The \textit{trace} of an element $x \in \FF_{q^2}$ is 
$$\text{Tr}x=x + \bar{x}=x + x^q.$$  
The norm and trace are surjective as maps $\FF_{q^2}^\times \rightarrow \FF_q^\times$ and $\FF_{q^2}^+ \rightarrow \FF_q^+$, respectively (Lemmas 12.1 and 12.5 in \cite{ps}).  The kernel of the norm map consists of all elements of the form $x\cdot \bar{x}^{-1} \in \FF_{q^2}^\times, \,\,\, x \in \FF_{q^2}^\times$ (Hilbert's Theorem 90, Corollary 12.2 in \cite{ps}).

A \textit{character} of $\FF_{q^2}^\times$ is a group homomorphism from $\FF_{q^2}^\times$ into $\CC^\times$.  We say a character $\theta$ is \textit{regular} (or \textit{primitive}, or \textit{non-decomposable}) if $\theta^q \neq \theta$.  A non-trivial character of $\FF_{q^2}^\times$ is regular if and only if it does not factor through the norm map $\text{N} : \FF_{q^2}^\times \rightarrow \FF_q^\times$ (Lemma 12.3 in \cite{ps}).  

\begin{lemma} \label{numreg} (part of Proposition 2.3 in \cite{knra})  Assume $q$ is odd.  Let $\nu$ be a given character of $\FF_q^\times$.  The number of regular characters of $\FF_{q^2}^\times$ that restrict to $\nu$ on $\FF_q^\times$ is $q-1$ if $\nu^\frac{q-1}{2} \equiv 1$, and $q+1$ if $\nu^\frac{q-1}{2} \not\equiv 1$.  
\end{lemma}

The fields $\FF_q$ are all of the (locally compact) finite fields.  Any field with the discrete topology is locally compact.  A familiar example of a locally compact \textit{non}-discrete field is $\RR$.  Eventually we will state the classification theorem for locally compact non-discrete fields.  First, we describe one example of a locally compact non-archimedean characteristic-$0$ non-discrete field:  the $p$-adic numbers, $\QQ_p$.

For $p$ prime, $r \in \QQ^\times$, write $r=p^k(a/b)$ with $p \nmid a,b$. 
The \textit{$p$-adic absolute value} 
$$\vert r \vert _p := p^{-k}, \qquad \vert 0 \vert _p :=0$$ 
satisfies
\begin{itemize}
\item[(i)] $\vert r \vert_p \geq 0;\,\, \vert r \vert_p = 0 \Leftrightarrow r=0$,
\item[(ii)] $\vert rs \vert_p = \vert r \vert_p \vert s \vert_p$, and
\item[(iii)] $\vert r+s \vert_p \leq \text{max} \lbrace \vert r \vert_p, \vert s \vert_p \rbrace \leq \vert r \vert_p + \vert s \vert_p$, 
\end{itemize} 
and it defines a metric $d_p(r,s)=\vert r-s \vert_p$ on $\QQ$. 

The \textit{$p$-adic numbers} are the completion of $\QQ$ with respect to $d_p$, denoted $\QQ_p$.  

\begin{theorem} (Theorem 4-12 in \cite{rv})
The locally compact non-discrete fields are $\RR$, $\CC$, $\QQ_p$ and its finite extensions, and the fields of formal Laurent series in one variable over a finite field.
\end{theorem}


The archimedean local fields are $\RR$ or $\CC$; the non-archimedean local fields of characteristic $0$ are $\QQ_p$ and its finite algebraic extensions; and the non-archimedean local fields of characteristic $p$ are the fields of formal Laurent series in one variable over a finite field (the quotient fields of $\FF_{p^n}[[t]]$).  The adjectives ``archimedean'' and ``non-archimedean'' refer to the absolute value on the field.

A quick comparison of $\QQ_p$ with $\RR$:  The only field automorphism of $\QQ_p$ is the identity; this is also true for $\RR$.  But the usual (archimedean) absolute value on $\RR$ does not satisfy the strong triangle inequality in item (iii) above.  As a consequence of the strong triangle inequality, the series $\sum _{n=0} ^\infty a_n$, $a_n \in \QQ_p$, converges in $\QQ_p$ if and only if $\lim _{n \rightarrow \infty} a_n = 0;$ on the other hand, the ``if'' fails in $\RR$.  More differences will become apparent as we outline the structure of $\QQ_p$.

The \textit{$p$-adic integers} are the completion of $\ZZ$ with respect to $d_p$, denoted $\ZZ_p$.  Equivalent definitions of $\ZZ_p$ are 
\begin{itemize}
\item the set $\lbrace x \in \QQ_p : \vert x \vert_p \leq 1 \rbrace = \lbrace x \in \QQ_p : \vert x \vert_p < p \rbrace$,
\item the maximal compact subring of $\QQ_p$,
\item the projective limit $\varprojlim_n \ZZ / p^n \ZZ$, and
\item the completion of the localization $\ZZ_{(p)}=\lbrace \frac{a}{b} \in \QQ : p \nmid b \rbrace$ with respect to $d_p$.
\end{itemize}

Elements of $\ZZ_p$ include $-1$, $p$, and $\frac{1}{1-p}$, but not $\frac{1}{p}$ (pg. 102 in \cite{neuk}).  Also, it can be shown that $\ZZ_p$ contains all the $(p-1)^{th}$ roots of unity (pg. 131 in \cite{neuk}).

The ideals in $\ZZ_p$ are 
$$\pp^n:=\lbrace x \in \QQ_p : \vert x \vert_p \leq p^{-n} \rbrace = \lbrace x \in \QQ_p : \vert x \vert_p < p^{-n+1} \rbrace, \,\,\,\, n \geq 0.$$
The unique maximal ideal in $\ZZ_p$ is 
$$\pp:=\pp^1 = \lbrace x \in \ZZ_p : \vert x \vert_p \leq p^{-1} \rbrace = \lbrace x \in \ZZ_p : \vert x \vert_p < 1 \rbrace = p \ZZ_p.$$
The \textit{residue field} of $\QQ_p$ is $\ZZ_p / \pp \cong \FF_p.$

We have a filtration 
$$\ZZ_p=\pp^0 \supset \pp^1 \supset \pp^2 \supset \cdots$$
and $\pp^n / \pp^{n+1} \cong \FF_p$.

Define 
$$U^n := 1 + \pp^n \text{   for   } n \geq 1,\text{     and     }U^0 := \ZZ_p^\times = \lbrace x \in \ZZ_p : \vert x \vert_p = 1 \rbrace.$$
Then we have a filtration 
$$\ZZ_p^\times = U^0 \supset U^1 \supset U^2 \supset \cdots$$
and $U^n / U^{n+1} \cong \FF_p^\times$.  (These filtrations and isomorphisms can be found on pg.\ 122 of \cite{neuk}.)

Some consequences of the above structure that will be relevant later in the project:

\begin{itemize}

\item $\QQ_p^\times \cong \lbrace p^n : n \in \ZZ \rbrace \times U^0 \cong \ZZ \times \FF_p^\times \times U^1.$

\item $GL(2,\ZZ_p)$ is a maximal compact subgroup of $GL(2,\QQ_p)$.  (Recall that $SO(2)$ is a maximal compact subgroup of $SL(2,\RR)$.)  

\item Lattices in $PGL(2,\QQ_p)$ are cocompact.  (Some lattices in $PSL(2,\RR)$, e.g. $PSL(2,\ZZ)$, are not cocompact.)

\item If $\Gamma$ is a lattice in $PGL(2,\QQ_p)$, then $\Gamma \backslash PGL(2,\QQ_p) / PGL(2,\ZZ_p)$ has a finite set of coset representatives.  (If $\Gamma$ is a lattice in $PSL(2,\RR)$, then $\Gamma \backslash PSL(2,\RR) / SO(2) \cong \Gamma \backslash \HHH$ has an an uncountable set of coset representatives.)  
\end{itemize}


$\ZZ_p$ is one example of a \textit{profinite group}: a topological group that is 
\begin{itemize}
\item[(i)] Hausdorff and compact, and 
\item[(ii)] admits a basis of neighborhoods of the identity consisting of normal subgroups --- $p^n\ZZ_p,$ in the case of $\ZZ_p$ --- or, equivalently, is totally disconnected (Definition 1.3 in \cite{neuk}).  
\end{itemize}

$\ZZ_p^\times$ is another example of a profinite group, with basis of neighborhoods of the identity given by $U^n$.  The groups $\QQ_p$ and $\QQ_p^\times$ (as well as the group $GL(2,F)$ introduced in the next chapter) are \textit{locally} profinite.  A group $G$ is called \textit{locally profinite} if every open neighborhood of the identity in $G$ contains a compact open subgroup of $G$.  

Ideal splitting over $\ZZ_p$ is quite different than over $\ZZ$.  Below are some examples of ideal splitting in the Gaussian integers $\ZZ[i]$, the ring of integers of the number field $\QQ (i)$.  
\begin{itemize}
\item $(2)=(1+i)^2,$ and we say the prime $2$ in $\ZZ$ \textit{ramifies} in $\ZZ [i]$, with \textit{ramification index} $2$.   \item $(13)=(2+3i)(2-3i),$  and we say the prime $13$ in $\ZZ$ \textit{splits} in $\ZZ[i]$.
\item $(7)=(7),$ and we say the prime $7$ in $\ZZ$ remains \textit{inert} in in $\ZZ[i]$, with \textit{inertial degree} $2$ (because $\ZZ[i]/(7)$ is a quadratic extension of the finite field $\ZZ / (7)$). 
\end{itemize}
  
In contrast, $\ZZ_p$ has a unique maximal ideal, generated by the prime element $p$.  So, there is only one prime to check, and we can say that a quadratic extension $E$ of $\QQ_p$ ($p \neq 2$) is itself ramified or unramified.  

Now we can talk about quadratic extensions of $\QQ_p$.  As we had for $\QQ_p$, we have for a quadratic extension $E$ of $\QQ_p$ a ring of integers $\OO_E$ (playing the role of $\ZZ_p$), a unique maximal ideal $\pp_E$ of $\OO_E$ (playing the role of $p\ZZ_p$), and a \textit{uniformizer} $\varpi_E$ that generates $\pp_E$ (playing the role of $p$).  Assume that $p \neq 2$, and let $\varepsilon$ denote a lift of a $(p-1)^{th}$ root of unity.

\begin{center} 
    \begin{tabular}{| c | c |}
    \hline
    unramified              &ramified       \\ \hline
    $E=\QQ_p(\sqrt{\varepsilon})$      &$E=\QQ_p(\sqrt{p})$ or $\QQ_p(\sqrt{p \varepsilon})$  \\  \hline
    $\varpi_E \OO_E = p \OO_E$      &$\varpi_E^2 \OO_E = p \OO_E$  \\ \hline
    $\OO_E / \pp_E \cong \FF_{p^2}$      &$\OO_E / \pp_E \cong \FF_{p}$  \\ 
    \hline 
    \end{tabular} 
\end{center}

(If $p=2$, we have $7$ quadratic extensions instead of $3$.)  Now let $F$ be any local non-archimedean field of characteristic zero, with residue field of order $q$, and assume $2 \nmid q$.  (So, $F$ is any finite extension of $\QQ_p$, where $p \neq 2$.)  We write $\OO_F$, $\pp_F$, $\varpi_F$, and $\OO_F/\pp_F \cong \FF_q$.  Writing $x\in F^\times$ (uniquely) as $x=u \varpi_F^n$, with $u \in \OO_F^\times$, we define 
$$\Vert x \Vert := q^{-n}.$$  
Let $E$ be a quadratic extension of $F$.  Then the table above still holds, with $\QQ_p$ replaced by $F$, and with $p$ replaced by $q$.  (The situation is more complicated when the order of the residue field of $F$ and the degree of the field extension $E/F$ are not relatively prime.)  

\begin{lemma} \label{TrN}
(18.1 Lemma in \cite{buhe})
Let $E/F$ be a quadratic extension, and assume $\OO_F / \pp_F \cong \FF_q$ with $2 \nmid q$.  Let $e$ denote the ramification index of $E$.
\begin{itemize}
\item[(i)] For $n \in \ZZ$, we have 
$$\emph{Tr}_{E/F}( \pp_E^{1+n}) = \pp_F^{1+\lfloor n/e \rfloor}=\pp^{1+n} \cap F.$$  
\item[(ii)]  For $n \geq 1$, we have
$$\emph{N}_{E/F}(1+x) \equiv \left( 1 + \emph{Tr}_{E/F}(x) \right) \emph{mod} \pp_F^{n+1} \qquad \left(x \in \pp_E^{en}\right)$$
and this map induces an isomorphism 
$$ U_E^{en} / U_E^{en+1} \xrightarrow{\sim} U_E^{n} / U_E^{n+1}. $$
\end{itemize}
\end{lemma}

Next we discuss characters of $F$ and $F^\times$.

\begin{proposition} \label{levchar}
(1.6 Proposition in \cite{buhe})
A group homomorphism from a locally profinite group into $\CC^\times$ has open kernel if and only if it is continuous.
\end{proposition}

By Proposition \ref{levchar}, if $\chi$ is a non-trivial character of $E^\times$, then there exists a smallest integer $m \geq 0$ such that $\chi$ is trivial on $U_E^{m+1}$ and non-trivial on $U_E^{m}$. We call $m$ the \textit{level} of $\chi$. Again by Proposition \ref{levchar}, if $\psi$ is a non-trivial character of $E$ (considered as an additive group), there exists a smallest integer $d$ such that $\psi$ is trivial on $\pp_E^d$ and non-trivial on $\pp_E^{d-1}$.  We call $d$ the \textit{level} of $\psi$.  Note the difference in the definition of level for multiplicative and additive characters.  Also, $E$ is the union of its compact open subgroups $\pp_E^n$; on the other hand, $E^\times$ has a unique maximal compact open subgroup, $U_E^1$.  This implies that all characters of $E$ are unitary, while characters of $E^\times$ need not be unitary.  Let $\phi$ be a fixed additive character of $E$ of level $d$; then all additive characters of $E$ are of the form $\phi_a(x)=\phi(ax), a\in E,$ and if $a\neq0$, the level of $\phi_a(x)$ is $d-m$, where $m$ comes from writing $a=\varpi^m x, \, x \in U_E^0$.  This justifies working with a fixed additive character of level $1$.

Lemma \ref{TrN} is used to prove

\begin{proposition}(18.1 Proposition in \cite{buhe})
Let $E/F$ be as above.  
\begin{itemize}
\item[(i)] Let $\psi$ be an additive character of $F$ of level $1$, and let $\psi_E=\psi \circ \emph{Tr}_{E/F}$. Then the character $\psi_E$ has level $1$.
\item[(ii)] Let $\chi$ be a character of $F^\times$ of level $n\geq 1$, and let $\chi_E = \chi \circ \emph{N}_{E/F}.$  Then the character $\chi_E$ has level $en$.  If $c \in \pp^{-n}$ satisfies 
$$\chi(1+x)=\psi(cx) \qquad \left(x \in \pp_F^{\lfloor n/2 \rfloor +1}\right),$$
then
$$\chi_E(1+y)=\psi_E(cy) \qquad \left(y \in \pp_F^{\lfloor en/2 \rfloor +1}\right).$$
\end{itemize}
\end{proposition}
Let $\chi$ be a character of $E^\times$.  We call the pair $(E/F,\chi)$ \textit{admissible} if 
\begin{itemize}
\item[(i)] $\chi$ does not factor through the norm map $\text{N}_{E/F} : E^\times \rightarrow F^\times$; and
\item[(ii)] if $\chi \vert_{U_E^1}$ does factor through $\text{N}_{E/F}$, then $E/F$ is unramified.
\end{itemize}

An admissible pair $(E/F, \chi)$ is said to be \textit{minimal} if $\chi \vert _{U_E^n}$ does not factor through $N_{E/F}$, where $n$ is the level of $\chi$.  We say two admissible pairs $(E/F,\chi)$ and $(E'/F,\chi ')$ are $F$-isomorphic if there exists an $F$-isomorphism $j:E \rightarrow E'$ such that $\chi '=\chi \circ j$.  

Let $\phi$ be a character of $F^\times$, and let $\phi_E = \phi \circ \text{N}_{E/F}$.  If $(E/F, \chi)$ is an admissible pair, then $(E/F, \chi \otimes \phi_E)$ is also an admissible pair.  Every admissible pair $(E/F, \chi)$ is isomorphic to an admissible pair $(E/F, \chi' \otimes \phi_E)$, for some $\phi$ a character of $F^\times$ and some minimal pair $(E/F,\chi')$. 






\chapter{Structure of $GL(2,\FF_q)$ and $GL(2,F)$} \label{strucGF}

Before discussing $GL(2,F)$, where $F$ is a local non-archimedean field, we begin with the structure of the finite group $GL(2,\FF_q), 2 \nmid q$.  Let
$$G_q=GL(2,\FF_q)= \left\lbrace g = \begin{pmatrix}
a & b \\
c & d
\end{pmatrix} \, : \, a,b,c,d \in \FF_q, \, ad-bc \neq 0 \right\rbrace.
$$
The standard Borel subgroup of $G$ is
$$B_q= \left\lbrace \begin{pmatrix}
  a     &  b\\ 
  0 & d 
\end{pmatrix} \, \in \,  G_q \right\rbrace .
$$
The unipotent radical of $B_q$, which may be identified with the additive group of $\FF_q$, is
$$N_q= \left\lbrace \begin{pmatrix}
  1     &  b\\ 
  0 & 1 
\end{pmatrix} \, \in \,  G_q \right\rbrace .
$$
The standard split maximal torus is
$$T_q= \left\lbrace \begin{pmatrix}
  a     &  0\\ 
  0 & d 
\end{pmatrix} \, \in \,  G_q \right\rbrace .
$$
The center of $G_q$, which we may identify with $\FF_q^\times$, is 
$$Z_q= \left\lbrace \begin{pmatrix}
  a     &  0\\ 
  0 & a 
\end{pmatrix} \, \in \,  G_q \right\rbrace .
$$

We have the semi-direct product decomposition 
$$B_q=T_q \ltimes N_q,$$ 
and the Bruhat decomposition 
$$G_q = B_q \cup B_q w B_q , \qquad w = \begin{pmatrix}
  0     &  1\\ 
  1 & 0 
\end{pmatrix} .$$


\begin{lemma} \label{GBindex}
$[ G_q : B_q ] = q+1.$
\end{lemma}
\begin{proof}
$\vert G_q \vert = (q^2-1)(q^2-q)$, $\vert B_q \vert = q(q-1)^2$, and we have $\vert G_q \vert / \vert B_q \vert = q+1$. 
\end{proof}


An eigenvalue $\lambda$ of $g \in G_q$ satisfies $gv-\lambda v=0$ for $v \in \FF_q \oplus \FF_q$ and is a solution to the quadratic equation $\text{Det}(g-\lambda I)=0$.  If one eigenvalue of $g$ lies in $\FF_q$, then so does the other eigenvalue, since they are both solutions to the same quadratic equation.  So, one possibility is that both eigenvalues lie in $\FF_q$.  The other possibility is that both eigenvalues lie in $\FF_{q^2}$, the unique quadratic extension of $\FF_q$.  In this case, $g$ is conjugate to an element of the form 
$$\begin{pmatrix}
  0     &  -\text{N}\alpha\\ 
  1 & \text{Tr}\alpha 
\end{pmatrix} $$
where N and Tr denote the norm and trace of a finite field extension at the beginning of Chapter \ref{lnaf}.  The irreducible representations corresponding to this conjugacy class of $G_q$ turn out to be the ``cuspidal'' representations, which we will introduce in the next chapter, and which can be constructed from regular characters of the quadratic extension $\FF_{q^2}^\times$, introduced at the beginning of the previous chapter.  (We included this discussion to indicate the interplay between trace, norm, determinant, characters, field extensions, and irreducible representations.  For a full explanation of how the determinant and trace separate conjugacy classes of $G_q$, see pg. 14 of \cite{ps}.)  

Now let $G=GL(2,F)$, where $F$ is a local non-archimedean field, and let $B$, $N$, $T$, and $Z$ be defined as were $B_q$, $N_q$, $T_q$, and $Z_q$, now over $F$ instead of $\FF_q$.  Next we define some other important subgroups of $G$.

The standard maximal compact subgroup of $G$ is
$$K= \left\lbrace \begin{pmatrix}
  a     &  b\\ 
  c & d 
\end{pmatrix} \, : \, a,b,c,d \in \OO_F, \, ad-bc \neq 0 \right\rbrace .
$$
The \textit{Iwahori subgroup} of $G$ is 
$$I= \left\lbrace \begin{pmatrix}
  a     &  b\\ 
  c & d 
\end{pmatrix} \, \in \,  G \, : \, a,d \in U_F^1, \, b \in \OO_F, \, c \in \pp_F  \right\rbrace .
$$

$K$ and $I$ are compact and open, and clearly $I$ is a subgroup of $K$.  Let us investigate the relationship between Haar measure on $K$ and $I$.  

\begin{lemma} \label{IKHaar}
If the Haar measure on $G$ (respectively, $G/Z$) is normalized to be $1$ on $I$ (respectively, $I.Z/Z$), then the Haar measure of $K$ (respectively, $K.Z/Z$) is $q+1$.  
\end{lemma}
\begin{proof}
The surjection $\OO_F \mapsto \OO_F / \pp_F \cong \FF_q$ induces a surjection $K \rightarrow GL(2,\FF_q)$, and the image of $I$ under this surjection is the standard Borel subgroup $B_q$, which by Lemma \ref{GBindex} has index $q+1$ in $GL(2,\FF_q)$.  So, $I$ has index $q+1$ in $K$, and $I.Z/Z$ has index $q+1$ in $K.Z/Z$; and measure scales with index.  
\end{proof}

The normalizer of the Iwahori subgroup has the form
$$N_G(I)= I. \langle \Pi \rangle, \qquad \Pi = \begin{pmatrix}
  0     & 1\\
  \varpi_F & 0 
\end{pmatrix},$$
and it properly contains $I.Z$.

A set of coset representatives for $I \backslash G / I$ is given by the affine Weyl group 
$$\mathbb{W}= \left\lbrace \begin{pmatrix}
  \varpi_F^s     &  0\\ 
   0& \varpi_F^t 
\end{pmatrix} \,\, \text{ or } \,\, 
\begin{pmatrix}
  0 &\varpi_F^s\\ 
  \varpi_F^t & 0
\end{pmatrix}
\, : \, s, t \in \ZZ  \right\rbrace
$$
(17.1 Proposition in \cite{buhe}).  The group $\mathbb{W}$ contains as a normal subgroup
$$\mathbb{W}_0= \left\lbrace x \in \mathbb{W} \, : \, \Vert x \Vert = 1 \right\rbrace,$$
and we have the semidirect product decomposition
$$ \langle \Pi \rangle \ltimes \mathbb{W}_0.$$
The set $S= \lbrace w, w' \rbrace$ generates the group $\mathbb{W}_0$ (17.8 Lemma 1 in \cite{buhe}), where 
$$w'=\Pi w \Pi^{-1} = \begin{pmatrix}
  0     & 1\\
  \varpi_F & 0 
\end{pmatrix}.$$
We may write $x$ in $\mathbb{W}_0$ as $x=w_1 w_2 \ldots w_r$ for $w_i \in S$.  Define the length of $x$, denoted $l(x)$, to be the smallest $r \geq 0$ for which such an expression exists.  We have $l(1)=0$; and for every integer $k \geq 1$, there are exactly $2$ elements of $\mathbb{W}_0$ with length $k$.  (This information about the affine Weyl group will be needed for the calculation of a formal dimension of a certain discrete series representation of $G$ later.)

The structure of some lattices in $G/Z = PGL(2,F)$ is given by

\begin{theorem} \label{iha} (Theorem 1 and Corollary in \cite{iha})
Let $F$ be a local non-archimedean field.  Any torsion-free discrete subgroup $\Gamma$ in $PGL(2,F)$ is isomorphic to a free group on at most countably many generators.  If moreover $\Gamma \backslash PGL(2,F)$ is compact, then the number of free generators of $\Gamma$ is equal to $\frac{1}{2}(q-1)h+1$, where $h = \vert \Gamma \backslash PGL(2,F) / PGL(2, \OO_F) \vert$, and $q$ is the order of the residue field of $F$.  
\end{theorem}

Note that the free group on $n$ generators can be found in $PGL(2,F)$ only if $h=2(n-1)/(q-1)$ is an integer, since $h$ is the cardinality of a (finite) set of double coset representatives.  As for existence, a construction is given in Section 4 of \cite{iha}.  Choosing $F=\QQ_3$, for example, would ensure that we can find all $F_n$ ($n \geq 2$) as lattices in $PGL(2,F)$.  

\begin{lemma} \label{covolFn}
Suppose $\Gamma$ is a lattice in $PGL(2,F)$ that is isomorphic to $F_n$, the free group on $n$ generators.  Then the volume of $PGL(2,F) / \Gamma$ is equal to
\begin{itemize}
\item[(i)] $\frac{2(n-1)}{q-1}$ if Haar measure on $PGL(2,F)$ is normalized so that the volume of $PGL(2,\OO_F)$ is $1$;
\item[(ii)] $\frac{2(n-1)(q+1)}{q-1}$ if Haar measure on $PGL(2,F)$ is normalized so that the volume of $PGL(2,\OO_F)$ is $q+1$; and
\item[(iii)] $n-1$ if Haar measure is normalized so that the volume of $PGL(2,\OO_F)$ is $\frac{1}{2}(q-1)$.
\end{itemize}
\end{lemma}
\begin{proof}
Like $SL(2,\RR)$ --- see Chapter \ref{setup} --- $GL(2,F)$ has an Iwasawa decomposition $N.A.K.$, where $N$ consists of the upper-triangular matrices with $1$ on the diagonal, $A$ consists of the diagonal matrices in $GL(2,F)$, and $K$ is the maximal compact subgroup $GL(2,\OO_F)$.  We may normalize the Haar measure $dg=dn.da.dk$ so that $\int_K dk=1$.  Because $GL(2,\OO_F)$ is open in $GL(2,F)$ (unlike $SO(2)$, which is not open in $SL(2,\RR)$), this also says $\int_G \mathbbm{1}_K dg=1$, where $\mathbbm{1}_K$ is the characteristic function of $K$.  We can carry out this normalization just as well for $GL(2,\OO_F).Z/Z \cong PGL(2,\OO_F)$ in $GL(2,F)/Z \cong PGL(2,F)$.
Observe $h = \vert \Gamma \backslash PGL(2,F) / PGL(2, \OO_F) \vert= 2(n-1)/(q-1)$.  So, if the Haar measure on $PGL(2,F)$ is taken to be $1$ on $PGL(2,\OO_F)$, then $\text{vol}(\Gamma \backslash PGL(2,F))$ is equal to $h=2(n-1)/(q-1)$.  We get the other two cases by applying Lemma \ref{IKHaar} and multiplying.  
\end{proof}


\chapter{Induced representations and the smooth dual}  \label{indual}

We begin with some definitions and theorems needed for the finite-dimensional representation theory of the finite group $GL(2,\FF_q)$.  (We thank Phil Kutzko for pointing out, at a very early stage of this project, that one cannot understand representation theory of $GL(2,F)$ without first understanding Frobenius reciprocity and Mackey theory in the finite-dimensional setting.)

Let $H$ be a subgroup of a finite group $G$, and let $(\sigma,W)$ be a finite-dimensional complex representation of $H$.  Let $X$ be the vector space of functions satisfying $f(hx)=\sigma(h)f(x)$ for $h \in H$ and $x \in G$.  Define a homomorphism $\Sigma : G \rightarrow \text{Aut}_\CC (X)$ by
$$ \Sigma(g) f : x \mapsto f(xg), \qquad (g,x \in G).$$
The pair $(\Sigma,X)$ is called the representation of $G$ \textit{induced} by $\sigma$, denoted by $\text{Ind}_H^G \sigma$.  

\begin{theorem} \label{finfrob}
(Frobenius reciprocity) Let $H$ be a subgroup of $G$, let $(\sigma, W)$ be a representation of $H$, and let $(\pi,V)$ be a representation of $G$.  Then
$$ \emph{Hom}_H (V,W) \cong \emph{Hom}_G (V, \emph{Ind}_H^G \sigma).$$
\end{theorem}




Information about the restriction of an induced representation is contained in

\begin{proposition} \label{finresind}
(Proposition 22 in \cite{jps}) 
The representation $\emph{Res}_K\emph{Ind}_H^G(W)$ is isomorphic to the direct sum of the representations $\emph{Ind}_{H_s} ( W_s ), s \in K \backslash G / H$, where $H_s$ denotes the subgroup $sHs^{-1} \cap K,$ and $W_s$ denotes the representation $\sigma( s^{-1} x s ), x \in H_s$. 
\end{proposition}

Proposition \ref{finresind} leads to ``Mackey's irreducibility criteria:''

\begin{proposition} (Proposition 23 in \cite{jps}) 
The necessary and sufficient conditions for the representation $\emph{Ind}_H^G W$ to be irreducible are
\begin{itemize}
\item[(i)] $W$ is irreducible, and
\item[(ii)]  For each $s \in G - H$, the two representations $\sigma( s^{-1} x s)$ and $\emph{Res}_{H_s}W$ are disjoint.
\end{itemize}
\end{proposition}

Let $\chi_1, \chi_2$ be characters of $\FF_q$, and define a character $\chi$ of $T_q$ by 
$$ \chi = \chi_1 \otimes \chi_2 :  
\begin{psmallmatrix} 
  a     & 0\\
  0 & d 
\end{psmallmatrix}
\mapsto \chi_1(a)  \chi_2(d).$$
We may consider $\chi$ to be a character of $B_q$, trivial on $N_q$, since $T_q \cong B/N$.   

Given $\chi$ defined above by $\chi_1$ and $\chi_2$, define the character
$$ \chi^w  :  
\begin{psmallmatrix} 
  a     & 0\\
  0 & d 
\end{psmallmatrix}
\mapsto \chi_2(a)  \chi_1(d).$$

Applying Mackey's irreducibility criteria, we obtain

\begin{proposition} \label{finchiw}
(6.3 Corollary 1 in \cite{buhe})  
Let $\chi$ be a character of $T_q$, viewed as a character of $B_q$ which is trivial on $N_q$.
\begin{itemize}  
\item[(i)]  The representation $\emph{Ind}_{B_q}^{G_q} \chi$ is irreducible if and only if $\chi \neq \chi^w$.
\item[(ii)]  If $\chi = \chi^w$, the representation $\emph{Ind}_{B_q}^{G_q} \chi$ has length $2$, with distinct composition factors.  
\end{itemize}  
\end{proposition}  

Let us examine how the induced representation $\text{Ind}_{B_q}^{G_q} \mathbbm{1}_{T_q}$ decomposes.  

\begin{proposition}  $\emph{Ind}_{B_q}^{G_q} \mathbbm{1}_{T_q}$ has dimension $q+1$, and 
$$ \emph{Ind}_{B_q}^{G_q} \mathbbm{1}_{T_q} = 1_{G_q} \oplus \emph{St}_{G_q}$$
for a unique irreducible representation $\emph{St}_{G_q}$.
\end{proposition}
\begin{proof}
We are inducing from the trivial representation of $B_q$, so by definition, the induced representation space $X$ must be $\CC[B_q \backslash G_q]$, which by Lemma \ref{GBindex} must have dimension $q+1$.  A quick application of Theorem \ref{finfrob} shows that $\text{Ind}_{B_q}^{G_q} \mathbbm{1}_{T_q}$ contains $\mathbbm{1}_{G_q}$:
$$ \text{Hom}_{G_q} (\text{Ind}_{B_q}^{G_q} \mathbbm{1}_{T_q}, \mathbbm{1}_{G_q}) \cong
\text{Hom}_{B_q} (\mathbbm{1}_{B_q},\mathbbm{1}_{G_q}),$$
which has dimension $1$.  Since $\mathbbm{1}_{T_q}=\mathbbm{1}_{T_q}^w$, we are in case (ii) of Proposition \ref{finchiw}, thus there are two composition factors of $\text{Ind}_{B_q}^{G_q} \mathbbm{1}_{B_q}$, one of which must be $\mathbbm{1}_{G_q}$.  The other is $\text{St}_{G_q}$.
\end{proof}

$\text{St}_{G_q}$ is called the \textit{Steinberg representation} of $G_q$.  The Steinberg representation is just one example of a representation of $G_q$ that satisfies the conditions below.

\begin{lemma}(6.3 Lemma in \cite{buhe}) \label{finN}
Let $\pi$ be an irreducible representation of $G_q$. The following conditions are equivalent:
\begin{itemize}
\item[(i)] $\pi$ is equivalent to a $G_q$-subspace of $\emph{Ind}_{B_q}^{G_q} \chi$, for some character $\chi$ of $T_q$;
\item[(ii)] $\pi$ contains the trivial character of $N_q$.  
\end{itemize}
\end{lemma}

We call an irreducible representations of $G_q$ that does \textit{not} satisfy the equivalent conditions above a \textit{cuspidal} representation.  These are the irreducible representations that cannot be constructed via induction alone.  An equivalent definition of a cuspidal representation of $GL(2,\FF_q)$ (given in the paragraph above Proposition 4.1.5 in \cite{bum}), which resembles our definiton of ``cuspidal functions'' given in Chapter \ref{autdef}, is that there exists no nonzero linear functional $\phi$ on the representation space $V$ such that
$$ \phi \left( \pi \begin{pmatrix}
1 &x \\
0 &1
\end{pmatrix} v \right) = \phi(v) \qquad (v \in V, \, x \in F ).$$

In Section 13 of \cite{ps}, the cuspidal representations of $G_q$ are constructed from regular characters of $\FF_{q^2}^\times$ (which we introduced at the beginning of Chapter \ref{lnaf}).

\begin{proposition} (Proposition 10.2 in \cite{ps}) \label{fincuspdim}
The dimension of a cuspidal representation of $G_q$ is $q-1$.
\end{proposition}

The finite-dimensional Steinberg and cuspidal representations above have analogues in the infinite-dimensional representation theory of $GL(2,F)$.  Before explaining those, we will need some background on smooth representations, compact induction, and duality. 

The material in the rest of this chapter is mostly drawn from \cite{buhe}.  We say a complex representation $(\pi,V)$ of $G$, 
$$\pi : G \rightarrow \text{Aut}_\CC (V),$$ 
is \textit{smooth} if, for every $v \in V$, there exists a compact open subgroup $K_v$ of $G$ such that $\pi(x)v=v$ for all $x \in K_v$.  This is equivalent to requiring that $V= \bigcup_K V^K$, where $V^K$ denotes the space of $\pi(K)$-fixed vectors in $V$, and $K$ ranges over the open compact subgroups of $G$.  We call a smooth representation $(\pi,V)$ \textit{admissible} if for each $K$, the space $V^K$ is finite-dimensional.

Note that the vectors in $V$ are functions from a locally profinite group (for example, $GL(2,F)$, where $F$ is a local non-archimedean field) into the field $\CC$, and the above definition of ``smooth'' has nothing to do with differentiability.  Compare this situation with the requirement in Chapter \ref{d} that an admissible $(\lag,K)$-module (for example, for $SL(2,\RR)$) be a Lie algebra representation, thereby consisting of vectors that are ``smooth'' in the usual sense meaning ``infinitely-differentiable;'' and the requirement that automorphic forms on $SL(2,\RR)$, which can be used to form a basis for certain admissible $(\lag,K)$-modules in $L^2(\Gamma \backslash G)$, be infinitely differentiable.  In both Chapters \ref{d} and \ref{discF}, smooth representations of a group are not Hilbert spaces, but they can be completed to Hilbert spaces affording unitary representations of the group.  Throughout Chapter \ref{discF}, we will do everything in the setting of smooth representations, waiting until Chapter \ref{Fex} to complete the representation space to the Hilbert space on which we will eventually represent a II$_1$ factor.

Let $(\pi,V)$ be a smooth representation of a locally profinite group.  Let $V^*$ denote the dual space $\text{Hom}_\CC (V,\CC)$ of $V$.  (If $V$ is finite-dimensional, then $V^* \cong V$.) Denote the canonical evaluation pairing by $V^* \times V \rightarrow \CC, \, (v^*,v) \mapsto \langle v^*, v \rangle.$  Consider the representation $\pi^*$ of $G$ defined by 
$$ \langle \pi^*(g)v^*,v \rangle = \langle v^*, \pi(g^{-1})v \rangle \qquad (g \in G, v^* \in V^*, v \in V)$$
Define
$$ \check{V} = \cup_K (V^*)^K. $$ 
(If $V$ is finite-dimensional, $V^* = \check{V}$.)  Now we have a smooth representation 
$$\check{\pi} : G \rightarrow \text{Aut}_\CC \check{V}$$ 
that we call the \textit{smooth dual} or \textit{contragredient} of $(\pi, V)$.  

By 2.8 Proposition in \cite{buhe}, restriction to $V^K$ induces an isomorphism $\check{V}^K \cong (V^K)^*$.  2.10 Proposition of \cite{buhe} states that if $(\pi,V)$ is admissible, then $(\pi,V)$ is irreducible if and only if $(\check{\pi},\check{V})$ is irreducible.  

Let $H$ be a subgroup of the locally profinite group $G$ (hence $H$ is also locally profinite), and let $(\sigma, W)$ be a smooth representation of $H$.  Let $X$ be the space of functions which satisfy $f(hg)=\sigma(h)f(g)$, $h\in H$, $g\in G,$ and for which there is a compact open subgroup $K_f$ of $G$ such that $f(gk)=f(g)$ for $g \in G$ and $k \in K_f$.  Define a homomorphism $\Sigma : G \rightarrow \text{Aut}_\CC (X)$ by
$$ \Sigma(g) f : x \mapsto f(xg) \qquad (g,x \in G).$$
The pair $(\Sigma,X)$ is a smooth representation of $G$, called the representation \textit{smoothly induced} by $\sigma$, and denoted by $(\Sigma,X)$ or $\text{Ind}_H^G \sigma$.

We have a functor $\text{Rep}(H) \rightarrow \text{Rep}(G)$ given by the map $\sigma \mapsto \text{Ind}_H^G \sigma$, and we have a canonical $H$-homomorphism 
\begin{align*}
\alpha_\sigma : \text{Ind}_H^G\sigma \rightarrow W, \\
f \mapsto f(1).
\end{align*}

Let $X_c$ denote the space of functions $f\in X$ that are compactly supported modulo $H$.  In other words, $\text{supp}f \subset HC$ for some compact set $C \subset G$.  Let $\text{c-Ind}_H^G$ denote the smooth representation on $X_c$. Again we have a functor $\text{Rep}(H) \rightarrow \text{Rep}(G)$, given by the map $\sigma \mapsto \text{c-Ind}_H^G \sigma$, and we have a canonical $H$-homomorphism 
\begin{align*}
\alpha_\sigma : \text{c-Ind}_H^G\sigma &\rightarrow W, \\
f &\mapsto f(1).
\end{align*}
This is known as \textit{compact induction}.  For any $G$, $H$, the morphism of functors $\text{c-Ind}_H^G \rightarrow \text{Ind}_H^G$ is an isomorphism if and only if $H\backslash G$ is compact.  

We have a version of Theorem \ref{finfrob} in the infinite-dimensional setting:

\begin{theorem}(2.4 in \cite{buhe})
Let $H$ be a closed subgroup of a locally profinite group $G$.  For a smooth representation of $(\sigma, W)$ of $H$ and a smooth representation $(\pi, V)$ of $G$, the canonical map
\begin{align*}
\emph{Hom}_G (\pi,\emph{Ind}_H^G \sigma ) &\rightarrow \emph{Hom}(\pi \vert H , \sigma), \\
\phi &\mapsto \alpha_\sigma \circ \phi,
\end{align*}
is an isomorphism that is functorial in both variables $\pi$ and $\sigma$.
\end{theorem}

Now let $G,B,N,$ and $T$ be as in Chapter \ref{strucGF}.  For a smooth representation $(\pi,V)$ of $G$, let $V(N)$ denote the subspace spanned by the vectors $v-\pi(n)v, \, n \in N, \, v\in V$, and let $V_N$ denote the space $V/V(N)$, which inherits a smooth representation $\pi_N$ of $B/N \cong T$.  The representation $(\pi_N,V_N)$ is called the \textit{Jacquet module} of $\pi$ at $N$.  

We have an infinite-dimensional version of Lemma \ref{finN}:

\begin{proposition} \label{jmnonz} (9.1 Proposition in \cite{buhe})
Let $\pi$ be an irreducible smooth representation of $G$. The following conditions are equivalent:
\begin{itemize}
\item[(i)] $\pi$ is equivalent to a $G$-subspace of $\emph{Ind}_B^G \chi$, for some character $\chi$ of $T$;
\item[(ii)] The Jacquet module of $\pi$ at $N$ is non-zero.  
\end{itemize}
\end{proposition}

Again, as in the finite case, an irreducible smooth representation of $G$ that does \textit{not} satisfy the equivalent conditions of Proposition \ref{jmnonz} is called \textit{cuspidal} or \textit{supercuspidal}.  

Proposition \ref{jmnonz} leads to an infinite-dimensional version of Proposition \ref{finchiw}:

\begin{theorem} (9.6 in \cite{buhe}) \label{infchiw}
Let $\chi=\chi_1 \otimes \chi_2$ be a character of $T$, and set $(\Sigma,X)=\emph{Ind}_B^G \chi$. 
\begin{itemize}
\item[(i)]  The representation $(\Sigma,X)$ is reducible if and only if $\chi_1 \chi_2^{-1}$ is either the trivial character or the character $x \mapsto \Vert x \Vert^2$ of $F^\times$.
\item[(ii)] Suppose $(\Sigma, X)$ is reducible. Then:
\begin{itemize}
\item[(a)] the $G$-composition length of $X$ is $2$;
\item[(b)] one composition factor of $X$ has dimension $1$, while the other is of infinite dimension;
\item[(c)] $X$ has a $1$-dimensional $G$-subspace if and only if $\chi_1\chi_2^{-1}=1$;
\item[(d)] $X$ has a $1$-dimensional $G$-quotient if and only if $\chi_1\chi_2^{-1}(x)=\Vert x \Vert^2, \, x\in F^\times$.
\end{itemize} 
\end{itemize}
\end{theorem}

In the setting of Theorem \ref{infchiw}, $\chi_1 = \chi_2$ if and only if $X$ has a $1$-dimensional $N$-subspace (part of 9.8 in \cite{buhe}).  The proof of Theorem \ref{infchiw} uses compact induction, as well as an infinite-dimensional version of  Proposition \ref{finresind}.


By (ii)(c) of Theorem \ref{infchiw}, the representation $\text{Ind}_B^G \mathbbm{1}_T$ has a $1$-dimensional $G$-subspace; so by (ii) (a) and (b), it has an irreducible $G$-quotient that we call the \textit{Steinberg representation} of $G$, denoted $\text{St}_G$.  It fits into the exact sequence
$$ 0 \rightarrow \mathbbm{1}_G \rightarrow \text{Ind}_B^G \mathbbm{1}_T \rightarrow \text{St}_G \rightarrow 0,$$ 
and it is in fact equivalent to a $G$-subspace of $\text{Ind}_B^G \chi$ for a certain character of $T$, as shown in 9.10 of \cite{buhe}. 

\chapter{Some discrete series representations of $GL(2,F)$} \label{discF}

(We will do everything in the setting of smooth representations, then complete these representations to obtain unitary representations later.)  

We use the notation of Chapter \ref{strucGF}.  A smooth irreducible representation $(\pi, V)$ of $G$ is a \textit{discrete series representation} if 
$$\int_{G / Z} \vert \langle \check{v},\pi(g)v \rangle \vert ^2 d\dot{\mu}(g) < \infty \qquad ( \check{v} \in \check{V}, \, v \in V),$$
where $\dot{\mu}(g)$ denotes a Haar measure on $G/Z$. 

\begin{theorem}[17.5 Theorem in \cite{buhe}] \label{Stsqint}
The Steinberg representation of $G$ 
(introduced in the last paragraph of Chapter \ref{indual})
has one matrix coefficient that is square-integrable.
\end{theorem}
\begin{proof}
The proof is long; most steps revolve around the Iwahori-Hecke algebra and a certain matrix coefficient $f(g)$ for which $f(\mathbbm{1})=1$, which we will not describe. The final line in the proof is
$$ \int_{G/Z} \vert f(g) \vert ^2 d \dot{g} = 2 \sum_{g \in \mathbb{W}_0 } q^{-l(g)},$$
with Haar measure on $G/Z$ normalized so that the volume of $I.Z/Z$ is $1$.  The series converges, by the discussion of possible lengths of elements in $\mathbb{W}_0$ in Chapter \ref{strucGF}. 
\end{proof}

As in the real semisimple setting --- see Proposition \ref{archSchur} --- a representation that is square-integrable modulo the center has a formal dimension $\text{d}_\pi$, defined by 
$$\int _{G/Z} \langle \check{\pi}(g)\check{v},v \rangle  \langle \check{v}',\pi(g)v' \rangle \mu(g) = \text{d}_{\pi}^{-1} \langle \check{v},v' \rangle \langle \check{v}',v \rangle \qquad (v,v' \in V, \, \check{v},\check{v}' \in \check{V}).$$

This is stated in Section 10a.2 of \cite{buhe}.  Exercise 1 at the end of Section 17 in \cite{buhe} asks:  \textit{What is the formal dimension of the Steinberg representation?}  We will give a solution to this exercise, which just amounts to looking at the last line of their proof that the Steinberg representation is square-integrable, expanding a geometric series, and keeping track of normalization of Haar measure.

\begin{proposition} \label{Stformdim}  
\begin{itemize}
\item[(i)] With Haar measure on $G/Z$ normalized to be $1$ on $I.Z/Z$, the formal dimension of the Steinberg representation is $\frac{1}{2}(q-1)(q+1)^{-1}$.  
\item[(ii)] With Haar measure on $G/Z$ normalized to be $1$ on $K.Z/Z$, the formal dimension of the Steinberg representation is $\frac{1}{2}(q-1)$.
\end{itemize}
\end{proposition}
\begin{proof}
Starting from the last line of the proof of Theorem \ref{Stsqint}, using the fact that lengths of elements of $\mathbb{W}_0$ are $0,1,1,2,2,3,3,\ldots$, as discussed in Chapter \ref{strucGF}, we expand the geometric series:
\begin{align*}
\int_{G/Z} \vert f(g) \vert ^2 d \dot{g} &= 2 \sum_{g \in \mathbb{W}_0 } q^{-l(g)} 
= 2 \left( 2 \left( \frac{1}{1-\frac{1}{q}} \right) -1 \right) \\
&= 2 \left( 2 \left( \frac{q}{q-1} \right) -1 \right) 
= 2 \left( \frac{2q}{q-1} - \frac{q-1}{q-1} \right) \\
&= 2 \left( \frac{q+1}{q-1} \right)
\end{align*}
Because the matrix coefficient $f(g)$ is equal to $1$ when $g=\mathbbm{1}$, we see from the definition of formal dimension that the formal dimension must be $\frac{1}{2}(q-1)(q+1)^{-1}$.  This is for the normalization of Haar measure on $G/Z$ assigning measure $1$ to $I.Z/Z$.  To obtain the formal dimension for the normalization of Haar measure on $G/Z$ assigning measure $1$ to $K.Z/Z$, Lemma \ref{IKHaar} says we must divide the previous measure by $q+1$, which corresponds to multiplying the formal dimension by $q+1$.  
\end{proof}

Part (ii) of Proposition \ref{Stformdim} agrees with equation (2.2.2) in \cite{cms}, which says 
$$ \text{d}_{\text{St}_G} \cdot \text{vol}(K.Z/Z) = \frac{1}{n} \prod_{k=1}^{n-1}(q^k-1),$$
where $\text{St}_G$ denotes the Steinberg representation of $GL(n,F)$ (with $n$ relatively prime to $q$), and $K$ is a maximal compact subgroup of $GL(n,F)$.

For cuspidal representations of $G$ (introduced after Proposition \ref{jmnonz} in Chapter \ref{indual}), we have something even better than square-integrability:

\begin{theorem} \label{cuspcpct} (10.1 and 10.2 of \cite{buhe})
The matrix coefficients of a cuspidal representation of $G$ are compactly supported on $G/Z$.
\end{theorem}
(This phenomenon of compactly-supported matrix coefficients cannot occur in the setting of Chapter \ref{d}.)

Section 3 of \cite{knra} and the beginning of Section 4.8 in \cite{bum} explain how to construct one of the simplest cuspidal representations of $GL(2,F)$, a ``depth-zero cuspidal representation,'' starting from a cuspidal representation of $GL(2,\FF_q)$ (introduced in Chapter \ref{indual}). Let $(\pi_\theta, V_\theta)$ be a cuspidal representation of $G_q$ built from the regular character $\theta$.  First, we may use the projection map $K \rightarrow GL(2,\FF_q)$ to lift $\pi_\theta$ to a representation of $K$.  

\begin{lemma} \label{centcharcusp} (stated at the beginning of Section 3 of \cite{knra}) The central character of this representation of $K$ is given by $z \mapsto \theta(z(1+\pp_F))$ for $z \in K \cap Z \cong \OO_F^\times$.  
\end{lemma}

Extend that central character of $Z \cap K \cong \OO_F^\times$ to a character of 
$$Z \cong F^\times = \bigcup _{n \in \ZZ} \varpi ^n \OO_F^\times$$
by specifying that the character takes the value $\alpha$ on $\varpi$, where $\alpha$ is any complex number with absolute value equal to $1$.  So, we may now view $\pi_\theta$ as a unitary representation of $ZK$.  Finally, let $(\pi, V)$ be the representation obtained by compactly inducing $\pi_0$ from $ZK$ to $G$:  
$$\pi = \text{c-Ind}_{ZK}^{G}(\pi_0).$$  
(Compact induction was introduced in Chapter \ref{indual}.)

\begin{proposition} (Proposition 1.2 in \cite{knra}, specialized to our situation) \label{indopendim}
For the depth-zero supercuspidal representation $\pi$ described above, the formal degree of $\pi$ is given by 
$$\emph{d}_\pi = \frac{\emph{dim}\pi_\theta}{\emph{vol}(ZK/Z)},$$ 
for any choice of Haar measure on $G/Z$.
\end{proposition}

\begin{proposition} \label{cuspformdim} Let $\pi$ be a depth-zero cuspidal representation of $G$.
\begin{itemize}
\item[(i)] If Haar measure on $G/Z$ is chosen so that $\emph{vol}(K.Z/Z)=1$, then $\emph{d}_\pi = q-1$.  
\item[(ii)]  If Haar measure on $G/Z$ is chosen so that $\emph{vol}(K.Z/Z)=\frac{1}{2}(q-1)$, then $\emph{d}_\pi = 2$. 
\end{itemize} 
\end{proposition}
\begin{proof}
By Proposition \ref{fincuspdim}, the dimension of $\pi_\theta$ is $q-1$.  Lemma \ref{IKHaar} gives the volumes to plug into Proposition \ref{indopendim}.  
\end{proof}

It is possible to obtain all cuspidal representations via compactly inducing representations of the open compact subgroups $K$ (which correspond to ``unramified'' cuspidal representations) and $N_G(I)$ (which correspond to ``ramified'' cuspidal representations); or via constructions starting from admissible pairs $(E/F,\chi)$ (defined in Chapter \ref{lnaf}), where the terms ``ramified'' and ``unramified'' describe the quadratic extension $E/F$.  We have only given one (unramified) example of a cuspidal representation.  

The non-cuspidal discrete series representations are called ``generalized special representations.''  These are twists of the Steinberg representation by a character; that is, representations of the form $\pi(g)=\varphi( \text{det}(g) ) \cdot \text{St}_G(g)$, for $\varphi$ a character of $F^\times$.

\chapter{Examples of von Neumann dimensions in a $\pp$-adic setting} \label{Fex} 

\begin{theorem} \label{vnapadicrep}
Let $G=PGL(2,F)$, where $F$ is a non-archimedean local field of characteristic $0$, with residue field of order not divisible by $2$.  Let $\Gamma$ be a torsion-free lattice in $G$, so that $\Gamma$ is a free group on $n$ generators, with $n$ finite.  Then there exist square-integrable unitary representations $(\pi_1, \HH_1)$ and $(\pi_2,\HH_2)$ of $G$ such that the restriction of each representation to $\Gamma$ extends to a representation of $R\Gamma$, with von Neumann dimension given by 
\begin{align*}
\emph{dim}_{R\Gamma}{\HH_i}=i(n-1) \qquad (i=1,2).
\end{align*}
\end{theorem}


\begin{proof}
First we explain the two irreducible unitary representations $(\pi_1, \HH_1)$ and $(\pi_2,\HH_2)$ of $G$; then, we say how the proof of Theorem \ref{discvnaext} in Chapter \ref{discvnaextpf} carries over to this setting, yielding  representations of $R\Gamma$ and the same formula for von Neumann dimension; and finally we calculate the von Neumann dimensions by computing the products of formal dimensions $\text{d}_{\pi_1}$ and $\text{d}_{\pi_2}$ and the covolume of $\Gamma$, using propositions and a lemma involving normalizations of Haar measure from the previous chapters.   

By Section 2 of \cite{cartier}, we may complete an irreducible smooth representation of $G$ to obtain an irreducible unitary representations of $G$.  Take $(\pi_1,\HH_1)$ to be the completion of the (smooth, admissible, irreducible) Steinberg representation of $GL(2,F)$, introduced in the last paragraph of Chapter \ref{indual}.  Then $\HH_1$ is an irreducible unitary representation of $GL(2,F)$; and by Theorem \ref{Stsqint}, $\HH_1$ is a subrepresentation of the right regular representation of $GL(2,F)$ on $L^2(GL(2,F)/Z)$.  By construction, the Steinberg representation is trivial on the center of $GL(2,F)$.  So, $\HH_1$ affords an irreducible unitary representation of $PGL(2,F)$ that is a subrepresentation of the right regular representation of $PGL(2,F)$ on $L^2(PGL(2,F))$.  

By Proposition \ref{numreg} in Chapter \ref{strucGF}, there are $q-1$ regular characters of $\FF_{q^2}^\times$ that are trivial on $\FF_q^\times$.  Choose any one of these to be $\theta$, so that by Proposition \ref{centcharcusp}, the corresponding representation of $K$ is trivial on the center of $K$.  In the construction of the (compactly-induced) depth-zero cuspidal representation following Proposition \ref{centcharcusp}, choose $\alpha$ to be $1$, so that the representation is trivial on the center of $GL(2,F)$.  Take $(\pi,\HH_2)$ to be the completion of this (smooth, admissible, irreducible) depth-zero supercuspidal representation.  By Theorem \ref{cuspcpct}, $\HH_2$ is a subrepresentation of the right regular representation of $GL(2,F)$ on $L^2(GL(2,F)/Z)$; and by choice of $\theta$ and $\alpha$, it factors through $PGL(2,F)$, so it is a subrepresentation of the right regular representation of $PGL(2,F)$ on $L^2(PGL(2,F))$.  

We have the main ingredients of the proof of Theorem \ref{discvnaext} in Chapter \ref{discvnaextpf}:  lattices with trivial center (the free groups on some finite number of generators guaranteed by Theorem \ref{iha} in Chapter \ref{strucGF}), and irreducible unitary subrepresentations of a right regular representation on a space of square-integrable functions.  The proof of Theorem \ref{discvnaext} carries over to this setting if we replace ``connected real semisimple Lie group without center'' (e.g. $PSL(2,\RR)$) with ``$PGL(2,F)$,'' and we arrive at the same formula
$$ \text{dim}_{R\Gamma}{\HH}=\text{d}_{\pi}\cdot \text{vol}(G/\Gamma). $$

Let Haar measure on $PGL(2,F)$ be normalized so that $\text{vol}(K.Z/Z)=1$.  With this normalization, $\text{d}_{\pi_1}=\frac{1}{2}(q-1)$, by (ii) of Proposition \ref{Stformdim} in Chapter \ref{discF}; $\text{d}_{\pi_2}=q-1$, by (i) of Proposition \ref{cuspformdim} in Chapter \ref{discF}; and $\text{vol}(\Gamma \backslash PGL(2,F))=\frac{2(n-1)}{q-1}$, by (i) of Lemma \ref{covolFn} in Chapter \ref{strucGF}.  Multiplying these together completes the proof.
\end{proof}

\chapter{The local Jacquet-Langlands correspondence} \label{lJLcorr} 

(We thank Winnie Li for pointing out, at a very early stage of this project, that information about formal degrees of discrete series representations of $GL(2,F)$ is contained in the local Jacquet-Langlands correspondence.)  

Let $D$ be a $4$-dimensional division algebra with center $\QQ_p$, where $p \neq 2$.  (We could let the center be any other non-archimedean local field of characteristic $0$ having residue field of order relatively prime to $2$, but we will stick with $\QQ_p$ for now).  $D$ is non-commutative; but like $\QQ_p$ and its finite extensions, $D$ has ``integers'' $\OO_D$, and a unique maximal prime ideal $\pp_D$.  

\begin{theorem} \label{quatstruc}(stated at the beginning of \cite{cordiv})
We can find elements $\pi, \alpha \in \OO_D$ such that 
\begin{itemize}
\item $D$ is generated over $\QQ_p$ by $\pi$ and $\alpha$, 
\item $\pi^2 \in \QQ_p$, and $\alpha$ is a root of unity, 
\item $\pi \OO_D = \OO_D \pi = \pp_D$,
\item $\lbrace 0, \alpha, \alpha^2, \ldots \alpha^{p^2-1} \rbrace$ is a complete set of residues for $\OO_D / \pp_D \cong \FF_{p^2}$, and
\item the map $\alpha \mapsto \pi^{-1} \alpha \pi $ generates the Galois group $\ZZ/2\ZZ$ of $\QQ_p[\alpha]$ over $\QQ_p$.
\item $\QQ_p[\alpha]$ is unramified, and $\QQ_p[\pi]$ is ramified.
\end{itemize}
\end{theorem}

$D^\times$ is compact modulo its center, so all its irreducible unitary representations are finite-dimensional.

\begin{theorem} (Theorem 15.1 in \cite{jl} and Corollary 4.4.5 in \cite{gl})
Let $F$ be a local field.  There exists a bijection between irreducible representations of the unit group $D^\times$ of the quaternion algebra $D$ over a local field $F$ and the discrete series representations of $GL(2,F)$, in which the central characters on both groups are the same, and the formal dimensions agree.  The Steinberg representation of $GL(2,F)$ corresponds to the trivial representation of $D^\times$.  If Haar measure on $GL(2,F)/Z$ is normalized so that the formal degree of the Steinberg representation is equal to $1$, then the formal dimensions of the discrete series representations of $GL(2,F)$ are equal to the actual dimensions of the irreducible complex finite-dimensional representations of $D^\times$.  
\end{theorem}

There is much more to the correspondence than this, but all we need are the correspondences between formal dimensions and central characters.
\vspace{3mm}

\noindent \textbf{Example.}  For $GL(2,\RR)$, we have 
$$D= \lbrace w + ix + jy + ijz \, : \, w, x, y, z \in \RR, \, i^2=j^2=-1, \, ij=-ji \rbrace,$$ 
the Hamiltonian quaternions; and 
$$D^\times \cong \RR _{>0} ^\times . SU(2).$$  
Irreducible representations of $SU(2)$ are of the form Sym$^n(\CC^2)$, so the formal dimensions of discrete series representations of $GL(2,\RR)$ are $1,2,3, \ldots$.  (This example and the next are discussed in the notes \cite{bcjl}.)
\vspace{3mm}

\noindent \textbf{Example.} For $GL(2,\QQ_p)$, $D$ is described by Theorem \ref{quatstruc}.  Let $\theta$ be a character of $\OO_D^\times / (1 + \pi \OO_D) \cong \FF_{p^2}^\times $ that doesn't factor through the norm map to $\FF_p^\times$ --- so, a regular character, defined at the beginning of Chapter \ref{strucGF}.  Extend $\theta$ to a character $\bar{\theta}$ of $\QQ_p^\times \times \OO_D^\times$ by requiring $p \mapsto \alpha \in \CC^\times$.  Since $\left[ D^\times : \QQ_p^\times \times \OO_D^\times \right]=2$, inducing $\bar{\theta}$ from $\QQ_p^\times \times \OO_D^\times$ to $D^\times$ gives us a $2$-dimensional representation of $D^\times$, which corresponds to a depth-zero cuspidal representation of $GL(2,\QQ_p)$, as discussed in Section 54.2 of \cite{buhe}.  Also, the trivial representation of $D^\times$ is $1$-dimensional, and it corresponds to the Steinberg representation of $GL(2,\QQ_p)$.  The normalization of Haar measure on $GL(2,\QQ_p)/Z$ in which the Steinberg representation has formal degree equal to $1$ is the normalization assigning to $K.Z/Z$ the measure $\frac{1}{2}(q-1)$, a consequence of Proposition \ref{Stformdim}.  So, the local Jacquet-Langlands correspondence provides a way of double-checking the two formal dimensions we calculated earlier, and the calculations agree.  
\vspace{3mm}

Without getting into the proof, we will just summarize why this correspondence works:  The determinant and trace separate conjugacy classes in $GL(2,F)$, while the ``reduced norm'' and ``reduced trace'' on $D$ (which restrict to the norm and trace on quadratic extensions of $F$ contained in $D$) separate conjugacy classes in $D^\times$; and  conjugacy classes are related to irreducible representations.  

The local Jacquet-Langlands correspondence gives us a unified way calculating formal dimensions of all discrete series representations of $GL(2,F)$ --- provided that we understand the representation theory of $D^\times$.  Exactly which representations of $D^\times$ are trivial on the center of $D^\times$, so that they correspond to representations of $GL(2,F)$ that factor through $PGL(2,F)$?  What is the full list of formal dimensions of discrete series representations of $GL(2,F)$ which factor through $PGL(2,F)$?  

The paper \cite{cms} uses a higher-dimensional generalization of the local Jacquet--Langlands correspondence (proved in \cite{rog} and \cite{dkv}) to give formal dimensions of discrete series representations of $GL(n,F)$, and actual dimensions of the finite-dimensional irreducible representations of degree $n$ over the field $F$, when $(n,q)=1$.  The formulas for these dimensions are given in terms of conductors of certain factorizations of the characters in minimal admissible pairs (which we introduced for $GL(2,F)$ at the end of Chapter \ref{lnaf}).  It turns out that we can construct a discrete series representation of $GL(2,F)$ from a minimal admissible pair, and we can construct a finite-dimensional irreducible representation of $D^\times$ from the same minimal admissible pair; and the formal dimension and the actual representation will be the same.  

Recall from Chapter \ref{lnaf} that $\QQ_p$ has three quadratic extensions, one unramified, two ramified; and that a character of the multiplicative group of a quadratic extension is described by its level.  So the possible minimal admissible pairs $(E/\QQ_p,\chi)$ are described by two pieces of information:  whether or not $E/\QQ_p$ is ramified, and the level of $\chi$.  Let $j$ denote the conductor of $\chi$, which is one more than the level.  If our calculations are correct, then $E$ ramified implies that $\chi$ must have even conductor, and Theorem 3.25 in \cite{cms} becomes
\begin{align*}
\text{d}_\pi = \begin{cases} 
      1                      & \text{$\pi$ a generalized special representation} \\ 
      2p^{j-1}, \, j=1,2,3, \ldots  & \text{$\pi$ an unramified cuspidal representation}  \\
      (p+1)p^{\frac{j-2}{2}}, \, j=2,4,6,\ldots   & \text{$\pi$ a ramified cuspidal representation}  \\
      \end{cases}
\end{align*}

According to the comments at the end of \cite{bh}, the product of the formal dimension of a cuspidal representation of $PGL(2,\QQ_p)$ and the covolume of a torsion-free (cocompact) lattice $\Gamma$ in $PGL(2,\QQ_p)$ equals the multiplicity of that cuspidal representation in the space $L^2(\Gamma \backslash PGL(2,\QQ_p)$.  (The archimedean version of this was Theorem \ref{dpigamma} at the end of Chapter \ref{discvnaextpf}.)  So if we can narrow down the list of formal dimensions above to those corresponding to cuspidal representations \textit{with trivial central character}, we will have calculated all possible multiplicities of cuspidal representations in $L^2(\Gamma \backslash PGL(2,\QQ_p)$.

\chapter{Future directions}

We plan to expand this dissertation and break it up into two papers:  

\begin{itemize}

\item[1.] The result proven in Chapter 9 is too specific to stand as a paper on its own.  We will obtain the most general results possible by using the limit multiplicity property in \cite{flm}, as well as the sources cited therein, to guarantee the occurrence of discrete series representations in interesting spaces; and we will combine this with the center-valued von Neumann dimension in \cite{bek} to work outside the setting of factors.  It may be worthwhile from the standpoint of von Neumann algebras to investigate the bimodule structure of these representation spaces.  

\item[2.] The methods used to prove the result in Chapter 14 apply only for two particular representations.  It would have been better to determine which formal dimensions in \cite{cms} correspond to representations with trivial central character; see the last paragraph of Chapter \ref{lJLcorr}.  This should be within reach after studying the local Langlands correspondence and subtleties of central characters in \cite{buhe}.  

\end{itemize}





\begin{ssp}
\setlength\bibitemsep{10pt}   
\printbibliography
\end{ssp}

\end{document}